\documentclass[11pt]{amsart}

 \usepackage{amssymb}
 \usepackage[colorlinks,linkcolor={blue}]{hyperref}
 \usepackage{graphicx}
 \usepackage{amscd}
 \usepackage{mathtools}
 \usepackage{enumerate}
  \usepackage{wasysym}
 \usepackage[dvipsnames]{xcolor}
 \usepackage{yhmath}
 \usepackage{psfrag}
 \usepackage{scalerel}

 \def\a{\alpha}
 \def\da{{\dot\alpha}}
 \def\be{\beta}
 \def\dbe{{\dot\beta}}
 
 \def\de{\delta}
 \def\De{\Delta}
 \def\e{\varepsilon}
 \def\deta{{\dot{\eta}}}

 \def\ga{\gamma}
 \def\dga{{\dot{\gamma}}}

 \def\Ga{\Gamma}
 
 \def\vr{\varphi}
 \def\vrt{\vartheta}

 \def\la{\lambda}
 
 \def\La{\Lambda}

 \def\si{\sigma}
 \def\Si{\Sigma}
 
 \def\Om{\Omega}

 \def\tt{\theta}
 \def\th{\theta}
 \def\Th{\Theta}

 \def\dxi{{\dot{\xi}}}

 \def\re{{\mathbb R}}
 \def\na{{\mathbb N}}

 \def\then{\Longrightarrow}
 \def\ov{\overline}
 \def\Z{{\mathbb Z}}

\def\Á{\textexclamdown}

 \def\A{{\mathbb A}}
 \def\cA{{\mathcal A}}

 \def\cC{{\mathcal C}}

 \def\cE{{\mathcal E}}

 \def\E{{\mathbb E}}

 \def\cH{{\mathcal H}}

 \def\L{{\mathbb L}}

 \def\Lo{{L}}

 \def\cM{{\mathcal M}}

 \def\cN{{\mathcal N}}

 \def\cO{{\mathcal O}}
 
 \def\cP{{\mathcal P}}

 \def\cU{{\mathcal U}}

 \def\cW{{\mathcal W}}

 \def\tq{{\tilde{q}}}
 
 \def\tq1{{\tilde{q}_1}}

 \def\dx{{\dot{x}}}

 \def\fX{{\mathfrak X}}

 \def\dy{{\dot{y}}}
 \def\dz{{\dot{z}}}

 \def\ogd{{\ov\ga_\de}}
 \def\gad{{\ga_\de}}
 
 \def\0{{\mathbf 0}}

 \def \lv{\left\vert}
 \def \rv{\right\vert}
 \def \lV{\left\Vert}
 \def \rV{\right\Vert}
 \def \ov{\overline}
 
 \def \then{\Longrightarrow}

 \definecolor{dgreen}{rgb}{0,0.3,0}
 \definecolor{dred}{rgb}{0.8,0,0}

 \def\ee{\text{\rm\large  e}}

 \DeclareMathOperator*{\tsum}{{\textstyle \sum}}
 \DeclareMathOperator{\supp}{supp}
 
 \DeclareMathOperator{\diam}{diam}

 \DeclareMathOperator{\intt}{int}

 \DeclareMathOperator{\Lip}{Lip}
 \DeclareMathOperator{\per}{per}

  \renewcommand{\proofname}{{\bf Proof:}}

 \theoremstyle{plain}

 \newtheorem{MainThm}{Theorem}
 \newtheorem{MainCor}[MainThm]{Corollary}

 \newtheoremstyle{Cl}
  {5pt}
  {3pt}
  {\sl}
  {}
  {\it}
  {:}
  {.5em}
  {}

 \newtheoremstyle{St}
  {5pt}
  {3pt}
  {\sl}
  {}
  {\bf}
  {.}
 {\newline}
  {}

 \def\begincproof{
                  \renewcommand{\proofname}{\it Proof:}
                  \begin{proof}
                 }

 \def\endcproof{
                \renewcommand{\qedsymbol}{$\diamondsuit$}
                \end{proof}
                \renewcommand{\qedsymbol}{\openbox}
                \renewcommand{\proofname}{\bf Proof:}
               }

  \swapnumbers

 \newtheorem{Thm}{Theorem}[section]
 
 \newtheorem{Lemma}[Thm]{\bf Lemma}
 
 \newtheorem{Corollary}[Thm]{\bf Corollary}
 \newtheorem{Theorem}[Thm]{\bf Theorem}
 \newtheorem{Proposition}[Thm]{\bf Proposition}

\swapnumbers
 \theoremstyle{Cl}
 \newtheorem{Claim}{Claim} [Thm]

\swapnumbers

 \theoremstyle{St}

 \theoremstyle{definition}
 \newtheorem{Definition}[Thm]{\bf Definition}

 \theoremstyle{remark}
 \newtheorem{Remark}[Thm]{\bf Remark}
 
 \newtheorem{mysec}[Thm]{}

\usepackage[colorinlistoftodos]{todonotes}
\makeatletter
\providecommand\@dotsep{5}
\makeatother
 \renewcommand{\proofname}{{\bf Proof:}}

 \title
 {Proof of the $C^2$ Ma\~n\'e's conjecture on surfaces.}

 \author[G. Contreras]{Gonzalo Contreras}

\address{CIMAT \\
          A.P. 402, 36.000 \\
          Guanajuato. GTO \\
          M\'exico.}
\email{gonzalo@cimat.mx}
\thanks{Partially supported by CONACYT, Mexico, grant  A1-S-10145.}


\begin{document}

\makeatother

\parskip +5pt

\begin{abstract}
We prove that $C^2$ generic hyperbolic Ma\~n\'e sets contain a periodic periodic orbit.
In dimension 2, adding a result in \cite{CFR}
which states that $C^2$ generic Ma\~n\'e sets are hyperbolic
we obtain  Ma\~n\'e's Conjecture for surfaces in the $C^2$ topology: 
Given a Tonelli Lagrangian $L$ on a compact 
surface $M$ there is a $C^2$ open and dense set of functions $f:M\to\re$ such
that the Ma\~n\'e set of the Lagrangian $L+f$ is a hyperbolic periodic orbit.
\end{abstract}

\maketitle

\tableofcontents


\section{Introduction.}

Let $M$ be a closed riemannian manifold.
A Tonelli Lagrangian is a $C^2$ function
\linebreak
 $L:TM\to\re$ that  is 
\begin{enumerate}[(i)]
\openup 5pt
\item {\it Convex:} $\exists a>0$\;
$\forall (x,v), (x,w)\in TM$, \;
$w\cdot \partial^2_{vv} L(x,v)\cdot w \ge a |w|_x^2$.
\end{enumerate}
The uniform convexity assumption and the compactness of $M$ imply that $L$ is
\begin{enumerate}[(i)]
\setcounter{enumi}{1}
\item {\it Superlinear:}
$\forall A>0$ $\exists B>0$ such that
$\forall (x,v)\in TM$: $L(x,v)> A\,|v|_x-B$.
\end{enumerate}

Given $k\in\re$, the Ma\~n\'e {\it action potential} is defined as
$\Phi_k:M\times M\to\re\cup\{-\infty\}$, 
\begin{equation}\label{actionpotential}
\Phi_k(x,y):=\inf_{\ga\in \cC(x,y)}
\int k+L(\ga,\dga),
\end{equation}
where 
\begin{equation}\label{defcxy}
\cC(x,y):=\{\ga:[0,T]\to M\; {\text{absolutely continuous }}\vert\;
T>0,\; \ga(0)=x,\;\ga(T)=y\;\}.
\end{equation}
The  Ma\~n\'e {\it critical value} is
\begin{equation}\label{defcL}
c(L) := \sup\{\, k\in \re \;|\; \exists\, x,y\in M :\;\Phi_k(x,y)=-\infty\;\}.
\end{equation}
See  \cite{CILib} for several characterizations of $c(L)$.

A curve $\ga:\re\to M$ is {\it semi-static} if 
$$
\forall s<t  \qquad
\int_s^t c(L)+L(\ga,\dga) = \Phi_{c(L)}(\ga(s),\ga(t)).
$$
Also $\ga:\re\to M$ is {\it static} if
$$
\forall s<t  \qquad
\int_s^t c(L)+L(\ga,\dga) = -\Phi_{c(L)}(\ga(t),\ga(s)).
$$
The {\it Ma\~n\'e set} of $L$ is
$$
\cN(L):=\{(\ga(t),\dga(t))\in TM \;|\; t\in\re,\; \ga:\re\to M\text{ is semi-static }\},
$$
and the {\it Aubry set } is
$$
\cA(L):=\{(\ga(t),\dga(t))\in TM \;|\; t\in\re,\; \ga:\re\to M\text{ is static }\}.
$$

The Euler-Lagrange equation
$$
\tfrac d{dt}\, \partial_vL = \partial_x L
$$
defines the Lagrangian flow $\vr_t$ on $TM$.
The {\it energy function} $E_L:TM\to\re$,
$$
E_L(x,v):= \partial_v L(x,v)\cdot v -L(x,v),
$$
is invariant under the Lagrangian flow.
The Ma\~n\'e set $\cN(L)$ is invariant under the Lagrangian flow and
it is contained in the energy level $\cE:=E_L^{-1}\{c(L)\}$ (see  Ma\~n\'e \cite[p.~146]{Ma7} or \cite{CILib}).

Let $\cM_{\text{inv}}(L)$ be the set of Borel probabilities in $TM$ which are invariant under the Lagrangian flow.
Define the {\it action } functional   $A_L:\cM_{\text{inv}}(L)\to\re\cup\{+\infty\}$ as 
$$
A_L(\mu):=\int L\,d\mu.
$$
The set of {\it minimizing measures} is
$$
\cM_{\min}(L):=\arg\min_{\cM_{\text{inv}}(L)} A_L,
$$
and the {\it Mather set } $\cM(L)$ is the union of the support of minimizing measures:
$$
\cM(L):=\textstyle\bigcup\limits_{\mu\in\cM_{\min}(L)} \supp(\mu).
$$
Ma\~n\'e proves (cf. Ma\~n\'e~\cite[Thm. IV]{Ma7} also \cite[p. 165]{CDI}) that an invariant measure is minimizing if and only if it is supported in the Aubry set. Therefore we get the set of inclusions
\begin{equation}\label{mane}
\cM\subseteq\cA\subseteq\cN\subseteq\cE.
\end{equation}

\begin{Definition}\label{defhipL}\quad

We say that $\cN(L)$ is {\it hyperbolic} if there are sub-bundles
$E^s$, $E^u$ of $T\cE|_{\cN(L)}$ and $T_0>0$ such that
\begin{enumerate}[(i)]
\item $T\cE|_{\cN(L)} = E^s \oplus\langle \frac d{dt}\vr_t\rangle \oplus E^u$.
\item $\lV D\vr_{T_0}|_{E^s}\rV < 1$, $\lV D\vr_{-T_0}|_{E^u}\rV <1$.
\item $\forall t\in\re$ \quad $(D\vr_t)^*(E^s)=E^s$, $(D\vr_t)^*(E^u)=E^u$.
\end{enumerate}
\end{Definition}

 Hyperbolicity for {\sl autonomous}
lagrangian or hamiltonian flows is  always understood as hyperbolicity
for the flow restricted to the energy level.

Fix a Tonelli Lagrangian $\Lo$. Let
$$
\cH^k(\Lo):=\{\,\phi\in C^k(M,\re)\;|\; \cN(\Lo+\phi) \text{ is hyperbolic }\},
$$
endowed with the $C^k$ topology.
By \cite[lemma 5.2, p. 661]{CP} the map $\phi\mapsto \cN(\Lo+\phi)$ is upper
semi-continuous and $\phi\mapsto c(\Lo+\phi)$ is continuous \cite[lemma 5.1]{CP}. 
This, together with  the persistence of hyperbolicity (cf. \cite[5.1.8]{FH} or  proposition~\ref{unifhip} below)
imply that  $\cH^k(\Lo)$ is an open set for any $k\ge 2$.

In \cite{CP} theorem C shows that generically $\cM=\cA=\cN$ is the support of a single 
minimizing measure.  Ma\~n\'e \cite[theorem F]{Ma6} proves that this measure is a strong limit of
invariant probabilities supported on periodic orbits.

Let 
$$
\cP^2(\Lo) := \{\, \phi \in C^2(M,\re) \;|\;
\cN(\Lo+\phi)\text{ contains a periodic orbit  or a singularity}\},
$$
and let $\ov{\cP^2(\Lo)}$ be its closure in $C^2(M,\re)$.
We will prove
\begin{MainThm}\label{HYP}
$\cH^2(\Lo)\subset \ov{\cP^2(\Lo)}$.
\end{MainThm}

In \cite{ham} we proved that if $\Ga\subset \cN(\Lo)$ is a periodic orbit, 
adding a potential $\phi_0\ge 0$ which is locally of the form $\phi_0(x) = \e\, d(x,\pi(\Ga))^2$
makes $\Ga$ a hyperbolic periodic orbit (or hyperbolic singularity) for the Lagrangian flow of $\Lo+\phi_0$ and
also $\cN(\Lo+\phi_0)=\Ga$. 
Moreover \cite[p.~934]{ham}, $\Ga$ has the {\sl locking property} meaning that there is a $C^2$ neighborhood $\cU$ of 
$\phi_0$ such that for $\phi\in\cU$, 
$\cN(L+\phi)=\Ga_\phi$ the
continuation $\Ga_\phi$ of the periodic orbit $\Ga$ in the energy level $E_{L+\phi}^{-1}\{c(L+\phi)\}$.
This follows from the semicontinuity of $\phi\mapsto \cN(L+\phi)$ and the expansivity of $\Ga$.
Therefore defining
$$
\cH\cP^2(\Lo):= \{\, \phi \in C^2(M,\re) \;|\;
\cN(\Lo+\phi)\text{ is a hyperbolic periodic orbit or singularity}\}
$$
we get
\begin{MainCor}
The set $\cH\cP^2(\Lo)$ contains an open and dense set in $\cH^2(\Lo)$.
\end{MainCor}

With A. Figalli and L. Rifford in \cite{CFR} we prove

\begin{MainThm}\label{CFR}
If $\dim M=2$ then
$\cH^2(\Lo)$ is open and dense.
\end{MainThm}

Thus for surfaces in the $C^2$ topology we obtain Ma\~n\'e's Conjecture
\cite[p. 143]{Ma7}:

\begin{MainCor}
If $\dim M =2$ then $\cH\cP^2(\Lo)$ contains an  open and dense set in $C^2(M,\re)$.
\end{MainCor}

Observe that from the inclusions in \eqref{mane}, for potentials $\phi\in \cH\cP^2(\Lo)$
the lagrangian $L+\phi$ has a unique minimizing measure and it is supported on a
hyperbolic periodic orbit or a hyperbolic singularity. The set $\cH\cP^2(\Lo)$
is open in the $C^2$ topology, so we can approximate the lagrangian $\Lo$ with a
$C^\infty$ potential $\phi$ to obtain a periodic minimizing measure, but 
the approximation is only proved to be $C^2$ small.

Since in theorem~\ref{HYP} the Aubry set is hyperbolic, by the shadowing lemma
$\cA(\Lo)$ is accumulated by periodic orbits. The idea of the proof is to choose a
special periodic orbit $\Ga$ nearby $\cA(\Lo)$ with small action and small period
and prove that adding a channel $\phi$ centered at $\Ga$, defined in \eqref{defphi}
produces that $\Ga\subset \cA(\Lo+\phi)$.

Theorem~\ref{HYP} is the same as the main theorem in the manuscript \cite{closed} which will
remain unpublished. The proof below uses a simplification devised by  
Huang, Lian, Ma, Xu, Zhang \cite{HLMXZ1},
\cite{HLMXZ2}, see also Bochi~\cite{Bochi3}.
The proof in \cite{closed}, \cite{ground}
 is based in the fact that generic hyperbolic
Ma\~n\'e sets have zero topological entropy.
The following proof is based on the periodic orbit
which is used to prove zero entropy.
The point is that the estimates of Bressaud and Quas~\cite{BQ} 
for the action and period of optimal periodic orbits nearby $\cA(\Lo)$
are so good that the cutting process in proposition~\ref{palga} stops
before the estimates get spoiled.

This proof owes a lot to the people working on ergodic optimization.
Ergodic optimization was born as a baby version of  Aubry-Mather 
theory adapted to symbolic dynamics \cite{CLT}. Now the subject has 
matured enough to give the main ideas of the proof of an important conjecture 
in Aubry-Mather theory.

Main differences of lagrangian systems with ergodic optimization besides that the
dynamical system depends on the lagrangian, are that perturbations need to be $C^2$
instead of Lipschitz and that the perturbations $\phi$ are defined in the configuration space 
and not in the phase space. These problems are solved by comparing the actions with
static orbits and using Fathi's differentiability estimates for weak KAM solutions, and 
observing that quasi minimizing objects inherit part of Mather's graph property.

In section~\ref{SS2} we obtain periodic specifications in $\cA(L)$ with 
exponentially small jumps and sub-exponential period. In section~\ref{SS3}
proposition~\ref{palga} we obtain periodic orbits $\Ga$ nearby $\cA(L)$ with small
action compared to their self-approximations,
called {\sl class I} by Yuan-Hunt~\cite{YH}.
In section~\ref{SS4} we prove in proposition~\ref{Ppert} that adding a 
channel $\phi$ centered in $\Ga$ we obtain $\Ga\subset \cA(L+\phi)$, 
thus proving theorem~\ref{HYP}.  In symbolic dynamics this was known
to Yuan-Hunt~\cite{YH} but here we use the method of Quas-Siefken~\cite{QS}.
In appendix~\ref{ashadowing} we prove 
the refinement of the shadowing lemmas that we need.

\section{Preliminars.}\label{SS1}

\openup -0.2pt

Let  $\cM_{\text{inv}}(L)$ be the set of Borel probabilities in $TM$ invariant under the 
 Lagrangian flow.
Denote by $\cM_{\min}(L)$ the set of minimizing measures for the Lagrangian $L$, i.e.
\begin{equation}\label{Mmin}
\cM_{\min}(L):=\Big\{\,\mu\in\cM_{\text{inv}}(L)\;\Big|\;
\int_{TM}L\,d\mu = -c(L)\;\Big\}.
\end{equation}
Their name is justified (cf. Ma\~n\'e \cite[Theorem II]{Ma7}) by 
\begin{equation}\label{minmeas}
-c(L) = \min_{\mu\in\cM_{\text{inv}}(L)}\int_{TM} L \; d\mu 
= \min_{\mu\in\cC(TM)}\int_{TM} L \; d\mu.
\end{equation}
 Fathi and Siconolfi \cite[Theorem 1.6]{FaSi} prove the second equality in \eqref{minmeas}
 where the set of  {\it closed measures} is defined by 
$$
{\mathcal C}(TM):=\Big\{ \, \mu \text{ Borel probability on }TM\; \Big\vert\;
\forall \phi\in C^1(M,\re) \; \int_{TM} d\phi\; d\mu =0\,\Big\}.
$$

Given a closed curve $\ga:[0,T]\to M$, using the closed measure 
$\int f \, d\mu_\ga := \frac 1T \int_0^T f(\ga,\dga) \, dt$ in
\eqref{minmeas} we get
\begin{equation}\label{cloga}
\ga \text{ closed curve in $M$}
\quad\then \quad
A_{L+c(L)}(\ga) \ge 0.
\end{equation}

Recall that a curve $\ga:\re\to M$ is {\it static} for a Tonelli Lagrangian $L$
if 
\begin{equation}\label{static2}
s<t \quad \then\quad
\int_s^t L(\ga,\dga) = -\Phi_{c(L)}(\ga(t),\ga(s));
\end{equation}
equivalently (cf. Ma\~n\'e \cite[pp. 142--143]{Ma7}), 
if $\ga$ is semi-static and 
\begin{equation}\label{aubry}
 s<t\quad\then\quad \Phi_{c(L)}(\ga(s),\ga(t))+\Phi_{c(L)}(\ga(t),\ga(s))=0.
\end{equation}
The {\it Aubry set} is defined as
$$
\cA(L):=\{\,(\ga(t),\dga(t))\;|\; t\in\re, \;\ga\text{ is static}\;\},
$$
its elements are called {\it static vectors}.
In this section we prove that with this definition $\cA(L)$ is invariant.

\openup+0.2pt

\pagebreak

\begin{Lemma}[A priori bound]\label{priori}\quad

For $C>0$ there exists $A_0=A_0(C)>0$ such that 
if $\ga:[0,T]\to M$
is a solution of the Euler-Lagrange equation with $A_L(\ga)<C\, T$, then 
$$
\lv\dga(t)\rv < A_0\qquad\text{ for all }t\in[0,T].
$$
\end{Lemma}

\begin{proof}
The Euler-Lagrange flow preserves the {\it energy function} 
\begin{equation}\label{EL}
E_L:= v \cdot \partial_v L -L.
\end{equation}
We have that
\begin{align}
\hskip -0.7cm\forall s\ge0\qquad
\tfrac{d\,}{ds}E_L(x,sv)\big|_s&=s\, v\cdot \partial_{vv}L(x,v)\cdot v \ge s\, a |v|_x^2.
\notag\\
E_L(x,v) &= E_L(x,0)+\int_0^{1} \tfrac{d\,}{ds}E_L(x,s v) \, ds
\notag\\
&\ge \min_{y\in M}E_L(y,0)+ \tfrac 12 a |v|_x^2.
\label{lbel}
\end{align}
Let
$$
g(r):= \sup\big\{w\cdot \partial_{vv}L(x,v)\cdot w \,:\, |v|_x\le r,\,|w|_x= 1\big\}.
$$
Then $g(r)\ge a$ and 
\begin{equation}\label{ubel}
E_L(x,v)\le \max_{z\in M}E_L(z,0) + \tfrac 12\, g(|v|_x)\, |v|_x^2.
\end{equation}

By the superlinearity there is $B>0$ such that $L(x,v)>|v|_x-B$ for all $(x,v)\in TM$.
Since $A_L(\ga)<C\,T$, the mean value theorem implies that there is $t_0\in]0,T[$ 
such that $|\dga(t_0)|<B+C$. Then~\eqref{ubel} gives an upper bound 
on the energy of $\ga$ and~\eqref{lbel} bounds the speed of $\ga$.

\end{proof}

 For $x,\,y\in M$ and $T>0$ define
 \begin{equation*}\label{defctxy}
 \cC_T(x,y):=
 \{\,\ga:[0,T]\to M\,|\, \ga \text{ is absolutely continuous}, \ga(0)=x,\,\ga(T)=y\,\}.
 \end{equation*}

 \begin{Corollary}\label{CAPB}\quad
 
 There exists $A_1>0$ such that if  $x,\,y\in M$ and $\ga\in\cC_T(x,y)$ is
 a solution of the Euler-Lagrange equation with 
 $$
 A_{L+c}(\ga) \le \Phi_c(x,y) + \max\{T,d(x,y)\},
 $$
 where $c=c(L)$,
 then
 \begin{enumerate}[(a)]
 \item $T\,\ge\,\tfrac 1A_1\; d(x,y)$.
 \label{CAPB1}
 \item  \;$|\dga(t)|\,\le\,A_1$ \; for all $t\in[0,T]$.
 \label{CAPB2}
 \end{enumerate}
 \end{Corollary}
 
 \begin{proof}
 First suppose that  $d(x,y)\le T$. Then item~\eqref{CAPB1} holds with $A_1=1$.
 Let
 \begin{equation}\label{llr}
 \ell(r):=|c|+\sup\{\,L(x,v)\,|\, (x,v)\in TM,\, |v|\le r\,\}.
 \end{equation}
 Since $d(x,y)\le T$, there exists a $C^1$ curve $\eta:[0,T]\to M$  joining $x$ to $y$ with $|\deta|\le 1$.
 We have that 
 $$
 A_{L+c}(\ga)\le \Phi_c(x,y)+T\le A_{L+c}(\eta)+T
 \le \big(\ell(1)+c\big)\,T+ T.
 $$
 Then item~\eqref{CAPB2} holds for $A_1=A_0(|\ell(1)+c+1|)$ where $A_0$ is from Lemma~\ref{priori}.
 
 Now suppose that $d(x,y)\ge T$.
 Let $\eta:[0,d(x,y)]\to M$ be a minimal geodesic with 
 $|\deta|\equiv 1$ joining $x$ to $y$. Let 
 $D:=\ell(1)+c +2>0$. From the superlinearity property there is 
 $B>1$ such that
 $$
 L(x,v) + c > D\, |v| - B, \qquad \forall (x,v)\in TM.
 $$
 Then
 \begin{align}
 [\ell(1)+c]\;d(x,y) 
    &\ge A_{L+c}(\eta) \ge \Phi_c(x,y)         
    \label{ap-1}\\
    &\ge A_{L+c}(\ga)-d(x,y)
    \label{ap-2}\\
    &\ge\int_0^T\bigl( D\;|\dga|-B\,\bigr)\,dt - d(x,y)
    \notag\\
    &\ge D\;d(x,y) - B\,T - d(x,y).
    \notag
 \end{align}
 Hence
 $$
 T \ge \tfrac{D-\ell(1)-c-1}{B}\; d(x,y) = \tfrac 1B \; d(x,y).
 $$
 This implies item~\eqref{CAPB1}.
 From~\eqref{ap-1} and~\eqref{ap-2}, we get that
 \begin{align*}
 A_L(\ga)&\le \bigl[\,\ell(1)+c+1\,\bigr]\;d(x,y)-c\,T ,
        \\
         &\le \bigl\{\,B\,[\,\ell(1)+c+1\,]-c\,\bigr\}\;T.
 \end{align*}
 Then Lemma~\ref{priori} completes the proof.
 
 \end{proof}

 We say that a curve $\ga:[0,T]\to M$ is a {\it Tonelli minimizer} if it minimizes the action functional
 on $\cC_T(\ga(0),\ga(T))$, i.e. if it is a minimizer with fixed endpoints and fixed time interval.

 \begin{Corollary}\label{Cbtm}
 There is $A>0$ such that if $x,\,y\in M$ and $\eta_n\in\cC_{T_n}(x,y)$, $n\in\na^+$ is a Tonelli minimizer
 with
 $$
 A_{L+c}(\eta_n)\le \Phi_c(x,y)+\tfrac1n,
 $$
 then there is $N_0>0$ such that $\forall n>N_0$, $\forall t\in[0,T_n]$, $|\deta_n(t)|<A$.
 \end{Corollary}

 \begin{proof}
 If $d(x,y)>0$ then for $n$ large enough $d(x,y)>\tfrac 1n$.
 In this case Corollary~\ref{CAPB} implies the result with the constant $A_1$.
 If $d(x,y)=0$ let $\xi_n:[0,T_n]\to\{x\}$ be the constant curve.
 Since $\eta_n$ is a Tonelli minimizer, we have that
 $$
 A_L(\eta_n)\le A_L(\xi_n) =\int_0^{T_n}L(x,0)\,dt\le |L(x,0)|\, T_n.
 $$
 Lemma~\ref{priori} implies that $|\deta_n|\le A_0(C)$ with $C=\sup_{x\in M}|L(x,0)|$.
 Now take
 \linebreak
  $A=\max\{A_0(C),A_1\}$.
  
 \end{proof}

\begin{Lemma}\label{ALinv}\quad

If $(x,v)$ is a static vector then $\ga:\re\to M$, $\ga(t)=\pi\vr_t(x,v)$ 
is a static curve, i.e. the Aubry set $\cA(L)$ is invariant.
\end{Lemma}

 \begin{proof}\quad
 
 Let $\ga(t)=\pi\,\vr_t(x,v)$ and suppose that $\ga|_{[a,b]}$ is static.
 We have to prove that all $\ga|_\re$ is static. 
 Let
 $\eta_n\in\cC_{T_n}(\ga(b),\ga(a))$
 be a Tonelli minimizer  with
  $$
 A_{L+c}(\eta_n)<\Phi_c(\ga(b),\ga(a))+\tfrac 1n.
 $$
 
 By Corollary~\ref{Cbtm}, for $n$ large enough,
 $|\deta_n|<A$. We can assume that 
 $\deta_n(0)\to w$. Let $\xi(s)=\pi\,\vr_s(w)$. If $w\ne \dga(b)$
 then for some $\e>0$ the curve $\ga|_{[b-\e,b]}*\xi|_{[0,\e]}$ is not $C^1$, and
 hence  it
 can not be a (Tonelli) minimizer of $A_{L+c}$ in 
 $\cC_{2\e}\big(\ga(b-\e),\xi(\e)\big)$. Thus
 $$
 \Phi_c(\ga(b-\e),\xi(\e))
        < A_{L+c}(\ga|_{[b-\e,b]})
        + A_{L+c}(\xi|_{[0,\e]}).
 $$
 \begin{align*}
 &\Phi_c(\ga(a),\ga(a))
    \le \Phi_c(\ga(a),\ga(b-\e))
         +\Phi_c(\ga(b-\e),\xi(\e))
         +\Phi_c(\xi(\e),\ga(a))
    \\
    &\;< A_{L+c}(\ga_{[a,b-\e]})
       +A_{L+c}(\ga|_{[b-\e,b]})
       +A_{L+c}(\xi|_{[0,\e]})
       +\liminf_n A_{L+c}(\eta_n|_{[\e,T_n]})
    \\
    &\;\le A_{L+c}(\ga|_{[a,b]})
      +\lim_nA_{L+c}\bigl(\,\eta_n|_{[0,\e]}*\eta_n|_{[\e,T_n]}\bigr)       
    \\
    &\;= -\Phi_c(\ga(b),\ga(a))+\Phi_c(\ga(b),\ga(a))
    =0.
 \end{align*}
 Thus there is a closed curve, from $\ga(a)$ to itself, with negative
 $L+c$ action, and also negative $L+k$ action for some $k>c(L)$.  
 Concatenating the curve with itself many times  
 shows that $\Phi_k(\ga(a),\ga(a))=-\infty$.
 By~\eqref{defcL} this implies that $k\le c(L)$, which is 
 a contradiction.
 Thus $w=\dga(b)$ and similarly $\lim_n\deta_n(T_n)=\dga(a)$.
 
 If $\limsup T_n<+\infty$, we can assume that $\tau=\lim_n T_n>0$
 exists. In this case $\ga$ is a semi-static
 periodic orbit of period $\tau+b-a$ and then $\ga|_\re$  is static.
 
 Now suppose that $\lim_n T_n=+\infty$.
 If $s>0$, we have that
 \begin{align*}
 A_{L+c}&(\ga|_{[a-s,b+s]})
   +\Phi_c(\ga(b+s),\ga(a-s))\le
   \\
   &\le\lim_n\big\{\,A_{L+c}(\eta_n|_{[T_n-s,T_n]})
       +A_{L+c}(\ga|_{[a,b]})
       \begin{aligned}[t]
       &+A_{L+c}(\eta_n|_{[0,s]})\,\big\}\\
       &+\Phi_c(\ga(b+s),\ga(a-s))
       \end{aligned}
    \\
    &\le\begin{aligned}[t]
    &-\Phi_c(\ga(b),\ga(a)) \\
    &+\lim_n \big\{\,A_{L+c}(\eta_n|_{[0,s]})
                    +A_{L+c}(\eta_n|_{[s,T_n-s]})
                    +A_{L+c}(\eta_n|_{[T_n-s,T_n]})
             \,\big\}
    \end{aligned}         
    \\
    &\le -\Phi_c(\ga(b),\ga(a))+\Phi_c(\ga(b),\ga(a))              
    =0.
 \end{align*}
 Thus $\ga_{[a-s,b+s]}$ is static for all $s>0$.

\end{proof}

\section{Optimal specifications.}\label{SS2}

Here lemma~\ref{Lper} and proposition~\ref{Pspec} follow
arguments by X. Bressaud and A. Quas~\cite{BQ}.

Let $A\in\{0,1\}^{M\times M}$ be a $M\times M$ matrix of with entries in $\{0,1\}$.
The subshift  of finite type $\Si_A$ associated to $A$ is the set
$$
\Si_A=\big\{\;(x_i)_{i\in\Z}\in \{0,1\}^\Z \;\big|\quad \forall i\in\Z\quad A(x_i,x_{i+1})=1\;\big\},
$$
endowed with the metric
$$
d(x,y) = 2^{-i}, \qquad i=\max\{\;k\in\na\;|\; x_i=y_i\;\; \forall |i|\le k\;\}
$$
and the {\it shift transformation}
$$
\si:\Si_A\to\Si_A, \qquad \forall i\in\Z\quad \si(x)_i = x_{i+1}.
$$

\bigskip

\begin{Lemma}\label{Lper}
Let $\Si_A$ be a shift of finite type  with $M$ symbols and topological
entropy $h$. Then $\Si_A$ contains a periodic orbit of period at most $1+M \ee^{1-h}$.
\end{Lemma}

\begin{proof}
Let $k+1$ be the period of the shortest periodic orbit in $\Si_A$.
We claim that a word of length $k$ in $\Si_A$ is determined by the set
of symbols that it contains.
First note that since there are no periodic orbits of period $k$ or less,
any allowed $k$-word must contain $k$ distinct symbols.
Now suppose that $u$ and $v$ are two distinct words of length $k$ 
in $\Si_A$ containing the same symbols. Then, since the words are different, 
there is a consecutive pair of   symbols, say $a$ and $b$, in $v$ which 
occur in the opposite order (not necessarily consecutively) in $u$.
Then the infinite concatenation of the segment of $u$ starting at $b$ 
and ending at $a$  gives a word in $\Si_A$ of period at most $k$, which 
contradicts the choice of $k$.

It follows that there are at most $\binom Mk$ words of length $k$.
Using the basic properties of topological entropy \cite[4.1.8]{LM1}
\begin{align*}
\ee^{hk}\le \binom Mk\le \frac {M^k}{k!}
\le \left(\frac {M \ee} k\right)^k.
\end{align*}
Taking $k$-th roots, we see that $k\le M\ee^{1-h}$.

\end{proof}

From now on we assume that the Ma\~n\'e set $\cN(L)$ is hyperbolic.
The definition of  a specification or pseudo-orbit appears in \ref{B8} in  appendix~\ref{ashadowing}.

\begin{Proposition}\label{Pspec}\quad

There are $C,\la>0$ such that for   $T>1$ large  there is 
$(\Th,\frak T)=(\{\th_i\},\{t_i\})$ a periodic T-specification in $\cA(L)$,
with $P=P_T$ jumps $(\th_i,t_i)=(\th_{i+P},t_{i+P})$,
and  period $\le 4T P_T$
such that 
\begin{gather}
\lim\nolimits_{T\to\infty}\tfrac 1T\log P_T =0,
\label{Psubexp}\\
\forall i\in\Z \mod P_T\qquad d\big (\vr_{t_i}(\th_{i}), \vr_{t_i}(\th_{i-1})\big )\le C\, \ee^{-\la T}.
\end{gather}

\end{Proposition}

\begin{proof}\quad

Given a specification $(\Th,\frak T)$ in $\cA(L)$ write
$\xi_i:[t_i,t_{i+1}]\to \cA(L)$, $\xi_i(s)=\vr_s(\th_i)$;
 and $\zeta_i:[0,t_{i+1}-t_i]\to \cA(L)$,
$\zeta_i(s)=\xi_i(s+t_i)$. We identify 
$(\Th,\frak T)$, $\{\xi_i\}$, $\{\zeta_i\}$
as the same specification. We will extend the 
definition of $\xi_i$, $\zeta_i$ to larger intervals,
with the same formula, as needed.

Let $T>0$ and let $\de>0$ be smaller than half of an expansivity constant \ref{dfe} for
$\cA(L)$ and smaller than $\be_0$ in proposition~\ref{B71} applied to $\cA(L)$.
Let $G=G_T$ be a minimal {\it $(2T,\de)$-spanning set} for $\cA(L)$, i.e.
\begin{equation}\label{gen}
\cA(L)\subset \bigcup_{\theta\in G}B(\theta,2T,\de),
\end{equation}
where $B(\theta,2T,\de)$ is the {\it dynamic ball}
$$
B(\theta,2T,\de)=\{\, \vrt\in TM\;|\; d(\vr_s(\theta),\vr_s(\vrt))\le\de\;\;\;\forall s\in[0,2T]\;\},
$$
and no proper subset of $G$ satisfies~\eqref{gen}.
Let $\Si\subset G^\Z$ be the bi-infinite subshift of finite type with symbols in $G$
and matrix $A\in \{0,1\}^{G\times G}$ defined by 
\begin{equation}\label{incA}
A(\theta,\vrt)= 1 
\qquad \Longleftrightarrow \qquad
\vr_{2T}(\theta)\in B(\vrt,2T,\de).
\end{equation}

Given $N\in\na^+$, let $S_N$ be a maximal $(2NT,2\de)$-separated set in $\cA(L)$, i.e.
\begin{equation}\label{sepset}
\tt,\,\vrt\in S_N,\quad \tt\ne\vrt \quad\then \quad
\vrt\notin B(\tt,2NT,2\de),
\end{equation}
and $S_N$ is a maximal subset of $\cA(L)$ with property~\eqref{sepset}.

Given $\tt\in S_N$ let $I(\tt)$ be an itinerary in $\Si$ corresponding to $\tt$, i.e.
$$
\forall n\in\Z\qquad \vr_{2nT}(\tt)\in B(I(\tt)_n,2T,\de), \qquad I(\th)_n\in G_T.
$$
If $\tt, \vrt\in S_N$ 
are different points,
then by \eqref{sepset} there are $0\le n<N$ and $s\in[0,2T]$ such that 
$d(\vr_{2nT+s}(\tt),\vr_{2nT+s}(\vrt))>2\de$. Thus $I(\tt)_n\ne I(\vrt)_n$,
i.e. $I(\tt)$, $I(\vrt)$ belong to different $N$-cylinders in $\Si$.
Therefore
$$
C_N := \#(\text{$N$-cylinders in $\Si$})
\ge \# S_N.
$$
Since $2\de$ is smaller than an h-expansivity  constant for $\cA(L)$,
see remark~\ref{rue},
its topological entropy can be calculated using $(n,2\de)$-separated 
(or $(n,\de)$-spanning) sets,
$h_{\rm top}(\vr,\cA(L))=h(\vr,\cA(L),2\de)$ (c.f.~Bowen~\cite{Bowen9} Thm. 2.4, p. 327),
thus
$$
h(\Si) \ge \limsup_N \frac{ \log \#C_N }N
\ge 2T \limsup_N \frac  {\log \# S_N}{2NT}
= 2T\,h_{top}(\cA(L)) =:2T  h.
$$
There is $K_T$ with sub-exponential growth in $T$ such that
$\#G_T\le K_T\,\ee^{2Th}$.
Then Lemma~\ref{Lper} gives a periodic orbit $\Th$ in $\Si$ with 
\begin{equation}\label{perTh}
P:={\rm period}(\Th)\le 1+ K_T\,\ee^{2Th}\; \ee^{1-2Th}\le 1 + K_T \,\ee.
\end{equation}

By Proposition~\ref{B71}, if $\tt,\,\vrt\in\cA(L)$ and $\vrt\in B(\tt,2T,\de)$ then
there is 
\begin{equation}\label{vtvrt}
|v|=|v(\vrt,\tt)|<D\,\de
\end{equation}
 such that 
\begin{equation}\label{cate1}
\forall \; |s|\le T\quad
d\big(\vr_{s+v+T}(\vrt),\vr_{s+T}(\tt)\big)
\le D\,\de\, \ee^{-\la(T -|s|)}.
\end{equation}

Given a sequence $(\tt_i)_{i\in\Z}\in\Si$, define a specification $(\zeta_i|_{[0,2T+v_i]})_{i\in\Z}$ in 
$\cA(L)$ by 
\linebreak
$v_i:=v(\vr_{2T}(\tt_i),\tt_{i+1})$ from~\eqref{vtvrt}, and $\zeta_i(s):=\vr_{s+T}(\tt_i)$. 
From~\eqref{incA} we have that 
$\vr_{2T}(\tt_i)\in B(\tt_{i+1}, 2T,\de)$.
Then by~\eqref{cate1}, with $\vrt=\vr_{2T}(\tt_i)$, $\tt=\tt_{i+1}$ and $s=0$
\begin{equation}\label{uijump}
d(\zeta_i(2T+v_i),\zeta_{i+1}(0))
= d(\vr_{3T+v_i}(\tt_i),\vr_T(\tt_{i+1}))
\le D\, \de\, \ee^{-\la T}.
\end{equation}

For the sequence $\Th\in\Si$ in \eqref{perTh} we have that $(\xi_i)_{i\in\Z}$ 
 a periodic $D \de\, \ee^{-\la T}$-possible specification
with $P$ jumps,  and period
\begin{equation}\label{uperiod}
{\rm period }(\{\xi_i\})\le (2T+D\de)\, (1+K_T\,\ee)\le 4T (1+ K_T\,\ee).
\end{equation}

\end{proof}

\section{Optimal periodic orbits.}\label{SS3}

A {\it dominated function} for $L$ is a function $u:M\to\re$ such that 
for any $\ga:[0,T]\to M$ absolutely continuous and $0\le s<t\le T$ we have
\begin{equation}\label{domi}
 u(\ga(t))-u(\ga(s))  \le \int_s^t \big[c(L)+L(\ga,\dga)\big]. 
\end{equation}
We say that the curve $\ga$ {\it calibrates} $u$ if the equality holds in \eqref{domi}
for every $0\le s<t\le T$. 
Dominated functions always exist, for example, by the triangle inequality for Ma\~n\'e's potential $\Phi_c$, the functions $u_p(x):=\Phi_c(p,x)$
are dominated for every $p\in M$. 
The definition of the Hamiltonian $H$ associated to $L$ implies 
that any $C^1$ function  
$u:M\to\re$ which satisfies
$$
\forall x\in M, \quad H(x,d_xu)\le c(L)
$$
is dominated.

\begin{Lemma}\label{statcal}
If $u$ is a dominated function and $\ga$ is a static curve
then $\ga$ calibrates $u$.
\end{Lemma}

\begin{proof}
Recall from \eqref{static2}, \eqref{aubry} that $\ga$ is static iff for all $s<t$ we have
\begin{equation}\label{staticd2}
\int_s^t\big[c(L)+L(\ga,\dga)\big] = -\phi_{c(L)}(\ga(t),\ga(s))=\phi_{c(L)}(\ga(s),\ga(t)).
\end{equation}
If $u$ is dominated, $\ga$ is static and $s<t$ we have that
\begin{align*}
u(\ga(t))\le u(\ga(s))+\phi_{c(L)}(\ga(s),\ga(t))
= u(\ga(s))-\phi_{c(L)}(\ga(t),\ga(s)).
\end{align*}
Using again the domination of $u$ and then the previous inequality we get
$$
u(\ga(s))\le u(\ga(t))+\phi_{c(L)}(\ga(t),\ga(s))\le u(\ga(s)).
$$
Therefore, using~\eqref{staticd2}, 
$$
u(\ga(t))=u(\ga(s))-\phi_{c(L)}(\ga(t),\ga(s))=
u(\ga(s))+\int_s^t\big[c(L)+L(\ga,\dga)\big].
$$
\end{proof}

\begin{Lemma}\label{domifathi}\quad

There are $K>0$ and $\de_0>0$ such that if $(z,\dz)\in\cA(L)$
 is a static vector, $u$ is a dominated function and $d(z,y)<\de_0$,
 then in local coordinates
 \begin{equation}\label{upvl}
 \big| u(y)-u(z)-\partial_vL(z,\dz)(y-z)\big|
 \le K\,|y-z|^2,
 \end{equation}
 where $y-z:=(\exp_z)^{-1}(y)$.

\end{Lemma}

\begin{proof}
Let $\E\subset TM$ be a compact subset such that 
$E^{-1}_L\{c(L)\}\subset\intt \E$. Cover $M$ by a finite 
set $\cO$ of charts. Fix $0<\e<1$  such that   if $\ga:[-\e,\e]\to M$
has velocity $(\ga,\dga)\in\E$ then $\ga([-\e,\e])$ lies inside the
domain of a chart in $\cO$.
There are $\de_1>0$ smaller than the Lebesgue
number of the covering $\cO$ and $A>0$ such that if $(x,v)\in\E$
and $\max\{|h|,\,|k|\}\le \de_1$ then in the charts
\begin{equation}\label{ldla}
\big| L(x+h,v+k)-L(x,v)-DL(x,v)(h,k)\big|\le A(|h|^2+|k|^2).
\end{equation}
Let $u:M\to\re$ be dominated and  $(z,\dz)\in\cA(L)$. Recall that 
$\cA(L)\subset E_L^{-1}\{c(L)\}\subset \E$. Write $\ga(t):=\pi\vr^L_t(z,\dz)$.
By Lemma~\ref{ALinv} the complete curve $\ga:\re\to M$ is static.
By Lemma~\ref{statcal}, $\ga$ calibrates $u$. Let $\de_0:=\e\,\de_1$.
Let $y\in M$ with $|y-z|<\de_0$ in a local chart. Define $\be:]-\e,0]\to M$ by
$$
\be(t):=\ga(t)+\left(\tfrac{t+\e}\e\right) (y-z).
$$
Then $\be(-\e)=\ga(-\e)$, $\be(0)=y$, $\dbe=\dga+\tfrac 1\e(y-z)$.
In particular $|\dbe-\dga|\le \tfrac 1\e|y-z|\le \de_1$ and we can apply~\eqref{ldla}.
\begin{align*}
\int_{-\e}^0 L(\be,\dbe)\le
\int_{-\e}^0L(\ga,\dga) 
+\int_{-\e}^0 
\Big\{L_x(\ga,\dga)(\be-\ga)
+L_v(\ga,\dga)(\dbe-\dga)\Big\}
+ A \e\big(1+\tfrac 1{\e^2}\big)\,|y-z|^2.
\end{align*}
Using that $\ga$ is a solution of the Euler-Lagrange equation
$\tfrac {d\,}{dt}L_v=L_x$ and integrating by parts, we get that
\begin{align}
\int_{-\e}^0L(\be,\dbe)&\le \int_{-\e}^0 L(\ga,\dga)\, dt+
L_v(\ga,\dga)(\be-\ga)\Big\vert_{-\e}^0 +  \tfrac {2A}{\e}\,|y-z|^2,
\notag\\
&\le \int_{-\e}^0 L(\ga,\dga)\, dt+ L_v(z,\dz)(y-z) + \tfrac{2A}{\e}\,|y-z|^2.
\label{lbdb}
\end{align}
Since $u$ is dominated and calibrated by $\ga|_{[-\e,0]}$ we obtain 
one of the inequalities in~\eqref{upvl}:
\begin{align*}
u(y) &\le u(\ga(-\e))+\int_{-\e}^0 c(L)+ L(\be,\dbe) 
\\
&\le u(\ga(-\e))+\int_{-\e}^0\Big\{L(\ga,\dga)+c(L)\Big\} dt+ L_v(z,\dz)(y-z) + \tfrac{2A}{\e}\,|y-z|^2
\\
&\le u(z) + L_v(z,\dz)(y-z) + \tfrac{2A}{\e}\, |y-z|^2.
\end{align*}

Now define $\a:[0,\e]\to M$ by
$$
\a(t):=\ga(t)+\left(\tfrac {\e-t}\e\right) (y-z).
$$
A similar argument to~\eqref{lbdb} gives
$$
\int_0^\e L(\a,\da)\, dt
\le \int_0^\e L(\ga,\dga) \,dt
-L_v(z,\dz)(y-z) + \tfrac{2A}{\e}\,|y-z|^2.
$$
Since $u$ is dominated we have that
\begin{align*}
u(\ga(\e))
&\le u(y)+\int_0^\e \Big\{L(\a,\da)+c(L)\Big\}
\\
&\le u(y) +\int_0^\e \Big\{L(\ga,\dga)+c(L)\Big\}\, dt
-L_v(z,\dz)(y-z) + \tfrac{2A}{\e}\,|y-z|^2.
\end{align*}
Since $u$ is calibrated by $\ga|_{[0,\e]}$ we have that
$$
u(\ga(\e))-\int_0^\e \Big\{L(\ga,\dga)+c(L)\Big\} = u(z).
$$
Thus we get the remaining inequality
$$
u(z)\le u(y)-L_v(z,\dz)(y-z)+ \tfrac{2A}{\e}\,|y-z|^2.
$$

\end{proof}

The set $\cN(L)$ is hyperbolic for the Euler-Lagrange flow restricted 
to the energy level $E_L^{-1}\{c(L)\}$.
There is a neighborhood $U$ of $\cN(L)$ in  $E_L^{-1}\{c(L)\}$ such that the 
set 
\begin{equation}\label{Lambda0}
\La =\textstyle \bigcap_{-\infty}^{+\infty}\vr_t(\ov U)
\end{equation}
 is hyperbolic, cf.~\cite[prop. 5.1.8] {FH}.
We can assume that $\cA(L)$ has no periodic orbits.
The neighborhood $U$ can be taken so small that any
periodic orbit $\Ga$ in $\La$ has period
\begin{equation}\label{pl10}
\per(\Ga)>10.
\end{equation}

For $B\subset TM$ write
$$
c(B,\cA(L)):=\sup\nolimits_{\th\in B}d(\th,\cA(L)).
$$

\begin{Proposition}\quad\label{palga}

For any $\e>0$ there is a periodic orbit $\Ga\subset\La\subset E_L^{-1}\{c(L)\}$, such that 
\begin{equation}\label{ALGA}
c(\Ga,\cA(L))< \e\,\ga(\Ga)\quad\text{and}\quad
A_{L+c(L)}(\Ga) < \e^2 \, \ga(\Ga)^2,
\end{equation}
where \qquad
$\ga(\Ga):=\min\{ d_{TM}(\Ga(s),\Ga(t)):|s-t|_{\rm mod\per(\Ga)}\ge 1\,\}$.
\end{Proposition}

\begin{proof}\quad

Let $T>0$ be very large which will be chosen at the end of the proof.
Let $\{\xi_i\}_{i=0}^{P-1}$, $\xi_i(t)=\vr_{t-t_i}(\th_i)$, $t\in[t_i,t_{i+1}[$  be 
the periodic specification  from proposition~\ref{Pspec}.
Define $(x_0,\dx_0):\re\to\cA(L)$ by $x_0(t)=\pi(\xi_i(t))$ if $t\in[t_i,t_{i+1}[$ and
$x_0(s+t_P-t_0)=x_0(s)$.

We will use repeatedly the constants from appendix~\ref{ashadowing} 
applied to the hyperbolic set $\La$ from \eqref{Lambda0}.
We will show that if $T$ is chosen sufficiently large then the objects at each step
are specifications and periodic orbits inside\footnote{Because they are (segments of) 
periodic orbits $\Ga$ with $c(\Ga,\cA(L))$ small, and hence $\Ga\subset U$.
Observe that $\La$ is not necessarily locally maximal, then a priori shadowing objects
could be outside $\La$.} $\La$.

By the shadowing theorem~\ref{SHL}, there is a periodic Euler-Lagrange solution
$(y_0,\dy_0)$ with energy $c(L)$ and a  continuous reparametrization 
$\si(t)$, with 
$| \si(t)-t|\le EC\, \ee^{-\la T}$ such that
$$
\forall t \qquad
d\big([x_0(t),\dx_0(t)],[y_0(\si(t)),\dy_0(\si(t))]\big) <E\cdot C\, \ee^{-\la T}.
$$

Then $Y_0(s):=(y_0(s),\dy_0(s))$ is a periodic orbit with a period near $\si(t_P)-\si(t_0)$.
We want a sequence of times $s_k$ nearby $\si(t_k)$ such that $s_P-s_0$ is
a period for $Y_0(s)$.  
Using canonical coordinates from~\ref{caco}
 define $w^k\in \re$ small by
\begin{align*}
\langle Y_0(\si(t_k)),\th_k\rangle 
&=W^{s}_\ga(Y_0(\si(t_k)))\cap W^{uu}_\ga(\th_k)
\\
&=W^{ss}_\ga\big(\vr_{w_k}(Y_0(\si(t_k)))\big)\cap W_\ga^{uu}(\th_k)
\ne\emptyset.
\end{align*}
Now let $s_k:=w_k+\si(t_k)$.
Observe that the time shift $w_k$ is determined by the sequence $\th_k$ which 
is periodic. Then the sequence $s_k$ is periodic with the period $s_P-s_0$ of $Y_0$
and by proposition~\ref{Pspec},
\begin{equation}\label{peryo}
\per(y_0):= {\rm period}(y_0)\le 5 T P_T.
\end{equation}

By proposition~\ref{B71} there is $D>0$  such that  
for $T$  large enough  
 there are $v^0_k$  such that
\begin{gather}
|v^0_k|\le D E\cdot C\ee^{-\la T},
\label{v0k}\\
\label{dx0y0}
\forall s\in[s_k,s_{k+1}]
\quad
d\big(Y^0(s),\vr_{s-s_i+v^0_k}(\th_i)\big)
\le D E\cdot C\,\ee^{-\la T} \, \ee^{-\la \min\{s-s_k,\,s_{k+1}-s\}}.
\end{gather}
Let $z^0_k(s):=\pi\vr_{s-s_i+v^0_k}(\th_i)$, $s\in[s_k,s_{k+1}]$.
Since by \ref{ALinv} $\cA(L)$ is invariant, we also have that $(z_k^0,\dz_k^0)\in\cA(L)$.

By adding a constant to $L$ we can assume that 
\begin{equation}\label{cloo}
c(L)=0.
\end{equation}
On local charts we have that
\begin{align*}
L(y_0,\dy_0) \le L(z^0_k,\dz^0_k) 
&+\partial_xL(z_k^0,\dz_k^0) (y_0-z^0_k) 
+\partial_vL(z^0_k,\dz^0_k) (\dy_0-\dz^0_k)
\\
&+ K_1 \, d\big([y_0(s),\dy_0(s)],[z^0_k(s),\dz^0_k(s)]\big)^2.
\end{align*}
Using that $z^0_k$ is an Euler-Lagrange solution we obtain
\begin{align}
\int_{s_k}^{s_{k+1}}L(y_0,\dy_0) 
&\le 
\left[\int_{s_k}^{s_{k+1}}L(z^0_k,\dz^0_k) \right]
+ \partial_vL(z_k^0,\dz_k^0)(y_0-z_k^0)\Big\vert_{s_k}^{s_{k+1}}
+
\notag\\
&+ K_1\int_{s_k}^{s_{k+1}} d\big([y_0(s),\dy_0(s)],[z^0_k(s),\dz^0_k(s)]\big)^2.
\label{iqK1}
\end{align}
Write $Z^0_k:=(z^0_k,\dz^0_k)$. Then
\begin{align}\label{ALz}
A_L(y_0) \le &\sum_{k=0}^{P_T-1} A_L(z^0_k)\;+
\notag\\
+&\sum_{k=0}^{P_T-1}\Big\{ \partial_v L(z^0_k,\dz^0_k)(y_0-z^0_k)\Big\vert_{s_{k+1}}
-\partial_v L(z^0_{k+1},\dz^0_{k+1})(y_0-z^0_{k+1})\Big\vert_{s_{k+1}}
\Big\}
\\
+&K_1 \sum_{k=0}^{P_T-1} \int_{s_k}^{s_{k+1}} d(Y_0,Z^0_k)^2 \, ds.
\notag
\end{align}
From \eqref{dx0y0}, for $K_2:=K_1 DEC$, the last term satisfies
\begin{equation}\label{siy0zk}
K_1 \sum_{k=0}^{P_T-1} \int_{s_k}^{s_{k+1}} d(Y_0,Z^0_k)^2 \, ds
\le
P_T\, K_2\,\ee^{-2\la T}.
\end{equation}

Let $u$ be a dominated function. By Lemma~\ref{domifathi},
if $(z,\dz)\in\cA(L)$ is a static vector then 
\begin{equation}\label{deuyz}
\big| u(y)-u(z) - \partial_vL(z,\dz)(y-z) \big|\le K_3\, |y-z|^2.
\end{equation}
By Lemma~\ref{statcal}, $u$ is necessarily calibrated by static curves.
Then using \eqref{deuyz},
\begin{align}
\sum_{k=0}^{P_T-1}A_L(z^0_k) &=\sum_k u(z^0_k(s_{k+1}))-u(z^0_k(s_k))
\notag \\
&=
\sum_k u(z^0_k(s_{k+1}))-u(z^0_{k+1}(s_{k+1}))
\notag \\
&= 
\sum_k
\big\{
u(z^0_k) -u(y_0) +u(y_0) -u(z^0_{k+1})\big\}\Big\vert_{s_{k+1}}
\notag \\
&\le \sum_k
\Big\{
\partial_v L(z^0_k,\dz^0_k)(z^0_k-y_0)
+
\partial_v L(z^0_{k+1},\dz^0_{k+1})(y_0-z^0_{k+1})\Big\}\Big\vert_{s_{k+1}}
+
\notag\\
&\hskip 2cm +K_3\,\Big\{|z^0_k-y_0|^2+|y_0-z^0_{k+1}|^2\Big\}
\Big\vert_{s_{k+1}}.
\label{ALxk0}
\end{align}
From  \eqref{dx0y0}  the last term satisfies
\begin{equation}\label{RALxk}
\sum_{k=0}^{P_T-1} K_3\,\Big\{|z^0_k-y_0|^2+|y_0-z^0_{k+1}|^2\Big\}
\Big\vert_{s_{k+1}}
\le
P_T\,K_4\,\ee^{-2\la T}.
\end{equation}
 Replacing  estimate  \eqref{ALxk0} for $\sum_k A_L(z_k^0)$ 
in inequality \eqref{ALz}
we obtain
\begin{gather}
A_L(y_0) \le  \sum_{k=0}^{P_T-1} \Big\{K_1\int_{s_k}^{s_{k+1}}d(Y_0,Z^0_k)^2\,ds
+K_3\, \big(|z^0_k-y_0|^2_{s_k}+|z^0_k-y_0|^2_{s_k+1}\big)\Big\}.
\label{ALY0int}
\intertext{Using  \eqref{siy0zk} and \eqref{RALxk} we have that }
A_L(y_0)\le\text{sum in \eqref{ALY0int}} \le K_5\, P_T\,\ee^{-2\la T}= : A_1(T).
\label{A1T}
\end{gather}
From \eqref{dx0y0} we get
\begin{equation}\label{cY0A}
c(Y_0,\cA(L)) < DE\cdot C\, \ee^{-\la T}.
\end{equation}
We can choose  in \eqref{A1T} $K_5>(DEC)^2$,  so that 
\begin{equation}\label{A1Tmax}
\max\{ A_L(y_0)^{\frac 12},\, c(Y_0,\cA(L))\} < A_1(T)^{\frac 12}.
\end{equation}
Also from \eqref{v0k},
\begin{equation}\label{v0k2}
|v^0_k|\le  A_1(T)^\frac12.
\end{equation}
If $\Ga=Y_0$ satisfies \eqref{ALGA} then the proof finishes.

If $\Ga=Y_0$ does not satisfy \eqref{ALGA}
then there are $r_1, r_2$,  $|r_1-r_2|_{{\rm mod }\,(s_P-s_0)}\ge 1$
such that 
\begin{equation}
\begin{aligned}
\e\, d(Y_0(r_1),Y_0(r_2)) &\le c(Y_0,\cA(L)), \qquad \text{or}
\\
\e^2\, d(Y_0(r_1),Y_0(r_2))^2 &\le A_{L}(y_0),
\label{dy012}
\end{aligned}
\end{equation}
using \eqref{cloo}.
Shifting the initial point of $Y_0$, we can assume that $r_1<r_2$ and 
$r_2-r_1 \le \tfrac 12 \per(y_0)$.
If for some $i,j$ we have that $|r_j-s_i|\le 1$ we replace $r_j$ by $s_i$
 and shift the other $r_k$ accordingly.
By Gronwall's inequality the distance $d(Y_0(r_1), Y_0(r_2))$ increases
at most by a multiple, say $B_0>1$. This insures that the times 
$\{ r_1,r_2,s_1,\dots s_{P-1}\}$ are all separated $({\rm mod}\, (s_P-s_0))$ at least by 1.
With this modification we get 
\begin{equation}\label{dr1r2}
r_2-r_1 \le \tfrac 12 \per(y_0)+2.
\end{equation}

In the following iteration process we will compare distances of a periodic orbit $Y_i(s)$
with a time shifted periodic orbit $Y_{i-1}(s+v^i)$. We will ensure in \eqref{tshift} that 
all the time shifts used are smaller than 1. We will take all the time shifts into account using
Gronwall's inequality by adding a multiple $B_0>1$ to the distance estimates.
Write
\begin{equation}\label{D0}
D_0:= B_0 \cdot D E>1.
\end{equation}

Let $Y_1=(y_1,\dy_1)$ be the closed orbit which shadows the periodic specification
$Y_0|_{[r_1,r_2]}$.
By~\eqref{cY0A}, for $T$ large, $Y_0\subset \La$ and hence by \eqref{pl10}, $\per(y_0)> 10$.
Then for $R=\frac 54$, using \eqref{peryo}, we have that 
\begin{equation}\label{pery1}
\per(y_1)\le
\tfrac 12\; \per(y_0) +3 \le R^{-1}\per(y_0)\le R^{-1} (5T P_T).
\end{equation}
By theorem~\ref{SHL}, proposition~\ref{B71} and \eqref{D0}, there is $v^1$ 
such that 
\begin{gather}
|v^1| \le D_0 \cdot d(Y_0(r_1),Y_0(r_2)),
\label{v1}\\
\forall s\in[r_1,r_2]\qquad
d(Y_1(s),Y_0(s+v^1))\le D_0 \, \ee^{-\la\min\{s-r_1,r_2-s\}}\,d(Y_0(r_1),Y_0(r_2)),
\label{dY0Y1}
\end{gather}
From \eqref{dy012} and \eqref{A1Tmax},
\begin{equation}\label{dY0Y1r}
d(Y_0(r_1),Y_0(r_2))\le \e^{-1}\,A_1(T)^\frac12.
\end{equation}

Using~\eqref{dY0Y1}, \eqref{dY0Y1r}, \eqref{dy012}, \eqref{A1Tmax} and $D_0\,\e^{-1}>1$,
\begin{align}
c(Y_1,\cA(L))&\le D_0\cdot d(Y_0(r_1),Y_0(r_2))
+ c(Y_0,\cA(L))
\notag
\\
&\le D_0\cdot\e^{-1} c(Y_0,\cA(L))
+ A_1(T)^\frac12
\notag\\
&\le 2 D_0\,\e^{-1} \,A_1(T)^{\frac 12}.
\notag\\
c(Y_1,\cA(L))&\le B_4\, A_1(T)^\frac12 \qquad\text{using \eqref{B3}.}
\label{cY1A}
\\
|v^1|+|v^0_k| &\le 
B_4 \,A_1(T)^{\frac 12}
\qquad \text{using \eqref{v1}, \eqref{v0k2}.}
\label{v1v0k}
\end{align}

In order to estimate the action of $Y_1$ we need to compare it
with a specification in $\cA(L)$.
Write $z^1_k(s)=z^0_k(s+v^1)$ and $Z^1_k=(z^1_k,\dz^1_k)$.
We cut the specification $\{Z^1_k|_{[s_k,s_{k+1}]}\}$
at $r_1$ and $r_2$ and remain with the periodic specification of
period $r_2-r_1\le \tfrac 12 \per(y_0)+2$, and jumps in $r_1<s_i<s_{i+1}<\ldots<s_j<r_2$
where $s_{i-1}\le r_1<s_i$ and $s_j<r_2\le s_{j+1}$.

 For $s\in[s_k,s_{k+1}]$ we have that 
\begin{align}
d(Y_1(s),Z_k^1(s)))&\le
d(Y_1(s),Y_0(s+v^1)) + d(Y_0(s+v^1),Z_k^1(s)).
\notag\\
d(Y_1(\cdot),Z_k^1(\cdot))^2 &\le
2\,d(Y_1(\cdot),Y_0(\cdot))^2 + 2\,d(Y_0(\cdot),Z_k^1(\cdot))^2,
\label{omitted1}\\
&\le 2\, (D_0)^2\, \ee^{-2\la \min\{s-r_1,r_2-s\}}\e^{- 2} \, A_1(T) \, +
\qquad\text{using~\eqref{dY0Y1}, \eqref{dy012}, \eqref{A1Tmax} }
\label{dy1y0}
\\
&\quad + 2\,(D_0 C)^2\,\ee^{-2\la T} \, \ee^{-2\la \min\{s-s_k,\,s_{k+1}-s\}},
\qquad\text{using \eqref{dx0y0}, \eqref{D0}.}
\notag
\end{align}
where the omitted arguments in \eqref{omitted1} are the same as in the previous inequality.

Repeating the estimates in~\eqref{iqK1}, \eqref{ALz}, ~\eqref{ALxk0} for the 
intervals between $r_1,s_i,\ldots,s_j,r_2$, we get
\begin{equation}
\begin{aligned}
A_L(y_1)&\le 
 \sum_{r_1\le s_k<r_2}\Big\{ K_1 \int_{s_{k}}^{s_{k+1}} d(Y_1,Z_k^1)^2\,ds
 +K_3
 \big( |y_1-z_{k}^1|^2_{s_k}
 +|y_1-z_{k}|^2_{s_{k+1}}\big)\Big\}.
\end{aligned}
 \label{aly1k3}
 \end{equation}

Using \eqref{omitted1} we can separate the sums in \eqref{aly1k3} into two sums.
The sums with terms $2\,d(Y_0,Z_k^1)^2$ or $2\,|y_0-z_k^1|^2$ are
about half of the terms in \eqref{ALY0int}, with a shift of $v^1$,
plus a term for the new jump at $(r_1,r_2)$.
The total number of jumps is $\le\frac 12 P_T+1\le P_T$, then 
the same estimate \eqref{A1T} gives:
\begin{align}
\sum_{r_1\le s_k<r_2}\int_{s_k}^{s_{k+1}}2 K_1\,d(Y_0,Z_k)^2 \,ds
+ 2K_3 \big(|y_0-z_k^1|^2_{s_k}+|y_0-z^1_k|^2_{s_{k+1}}\big)
\le 2\,A_1(T).
\label{AY1a}
\end{align}

The other sum uses terms with  $2\,d(Y_1(s),Y_0(s+v^1))^2$ which are
bounded in~\eqref{dy1y0}. Abbreviating the time shift $v^1$, this sum writes
\begin{align}
\sum_{r_1\le s_k<r_2}&\int_{s_k}^{s_{k+1}}2 K_1\,d(Y_1,Y_0)^2 \,ds
+ 2K_3 \big(|y_1-y_0|^2_{s_k}+|y_1-y_0|^2_{s_{k+1}}\big)
\notag\\
&\le \int_{r_1}^{r_2} 2 K_1\,d(Y_1,Y_0)^2 \,ds
+
\sum_{r_1\le s_k<r_2}
2K_3 \big(|y_1-y_0|^2_{s_k}+|y_1-y_0|^2_{s_{k+1}}\big)
\notag\\
&\le
2 K_1\,(D_0)^2 B_1\,\e^{-2} \, A_1(T) + 2 K_3\, (D_0)^2\, 2 B_2\, \e^{-2} A_1(T)
\notag\\
&\le 
B_3\, A_1(T),
\label{AY1b}
\end{align}
using~\eqref{dY0Y1}, \eqref{dY0Y1r}, where 
\begin{gather}
\textstyle
B_1:=\int_{-\infty}^{+\infty} \ee^{-2\la |s|}ds= \frac 1\la,
\qquad
B_2:= 1+\sum_{n\in\na} \ee^{-2\la n },
\notag\\
B_3 =B_3(\e):= 2\, (K_1+K_3)\,(D_0)^2 (B_1+2 B_2)\, \e^{-2}.
\label{B2}
\end{gather}
Adding \eqref{AY1a} and \eqref{AY1b} we get
\begin{equation}
A_L(y_1)\le \text{sum in \eqref{aly1k3}}\le B_4 \, A_1(T),
\label{AY1}
\end{equation}
where 
\begin{equation}\label{B3}
B_4:=\max\{B_3+4, 2D_0\,\e^{-1}\}>4.
\end{equation}

If $Y_1=\Ga$ satisfies \eqref{ALGA} the proof finishes.
If not there are $r_3<r_4$, $r_4-r_3\le \tfrac 1{2}\per(y_1)+2$,
such that 
\begin{align}
d(Y_1(r_3),Y_1(r_4)) &\le\e^{-1} \max\{ A_{L}(y_1)^\frac12, c(Y_1,\cA(L))\}
\notag\\
&\le \e^{-1} B_4\, A_1(T)^\frac12 \qquad\text{using \eqref{AY1}, \eqref{cY1A}.}
\label{cr3r4}
\end{align}
We shadow the specification $Y_1|_{[r_3,r_4]}$ by a periodic orbit
$Y_2=(y_2,\dy_2)$ with
\begin{gather}
|v^2| \le D_0\cdot d(Y_1(r_3),Y_1(r_4)),
\label{v2}\\
\forall s\in[r_3,r_4]\qquad
d(Y_2(s),Y_1(s+v^2))\le D_0 \, \ee^{-\la\min\{s-r_3,r_4-s\}}\,d(Y_1(r_3),Y_1(r_4)),
\notag\\
\per(y_2)\le \tfrac 12\per(y_1)+3\le R^{-2} \per(y_0).
\notag
\end{gather}
Then
\begin{align*}
c(Y_2,\cA(L)) 
&\le D_0\cdot d(Y_1(r_3),Y_1(r_4)) + c(Y_1,\cA(L))
\\
&\le 2D_0 \, \e^{-1} B_4\, A_1(T)^\frac12
\qquad\text{using \eqref{cr3r4}, \eqref{cY1A}}
\\
&\le (B_4)^2\, A_1(T)^{\frac12}
\qquad\text{using \eqref{B3}.}
\\
|v^2|+|v^1|+|v^0_k|
&\le (B_4)^2\, A_1(T)^{\frac12} \qquad\text{similarly, using \eqref{v2}, \eqref{v1v0k}.}
\end{align*}

We need to compare $Y_2$ with a specification in $\cA(L)$.
Write $z_k^2(s)=z_k^1(s+v^2)$ and $Z_k^2=(z_k^2,\dz_k^2)$.
Then

\begin{align}
A_L(y_2)&\le 
\sum_{r_3\le k<r_4}   \int_{s_{k}}^{s_{k+1}} K_1 \,d(Y_2,Z_k^2)^2 \,ds
 +K_3\,\big( |y_2-z_{k}^2|_{s_{k}}^2
 +|y_2-z_{k}^2|_{s_{k+1}}^2\big).
 \label{aly2k3}
\end{align}
\begin{equation}\label{syszk}
d(Y_2,Z^2_k)^2\le 2\,d(Y_2,Y_1)^2 + 2\, d(Y_1,Z^2_k)^2.
\end{equation}
\begin{align}
\int_{r_3}^{r_4} 2K_1\, d(Y_2,&Y_1)^2 \,ds
+ \sum_{r_3\le s_k< r_4} 2K_3\big(|y_2-y_1|^2_{s_k}+|y_2-y_1|^2_{s_{k+1}}\big)
\label{sum30}\\
&\le
\big\{2K_1 \,B_1\, (D_0)^2+ 2K_3\,(D_0)^2 \,2 B_2\big\} \;d(Y_1(r_3),Y_1(r_4))^2
\notag\\
&\le
2 (K_1+K_3) (D_0)^2 (B_1+2B_2) \,\e^{-2}  (B_4)^2 A_1(T)
\qquad\text{using~\eqref{cr3r4}},
\notag\\
&\le B_3\, (B_4)^2\,A_1(T)
\qquad\text{using~\eqref{B2}.}
\label{B2B3}
\end{align}
From \eqref{syszk} and \eqref{aly2k3} we have that 
\begin{align*}
AL(y_2)
&\le
\text{sum in~\eqref{sum30}}+2\,\text{sum in~\eqref{aly1k3}}
\\ 
&\le
(B_4)^3\, A_1(T)
\qquad\text{using \eqref{B2B3}, \eqref{AY1}, \eqref{B3},}
\\
&\le (B_4)^4 A_1(T).
\\
\per(y_2) &\le R^{-2}  \per(y_0) \le R^{-2}  (5T P_T).
\end{align*}

At the $n$-th iteration we have
\begin{align}
A_L(y_n)&\le (B_4)^{2n} \, A_1(T),
\notag\\
\per(y_n)
&\le R^{-n} \per(y_0)\le R^{-n} (5TP_T).
\label{peryn}
\end{align}
\begin{gather*}
c(Y_n,\cA(L)) 
\le 
(B_4)^n A_1(T)^\frac12.
\\
|v^0_k|+\tsum_{i=1}^n|v^i|
\le
(B_4)^n A_1(T)^\frac12.
\end{gather*}

Let $\a_2>0$ be such that 
$\{\,\th\in T^*M: d(\th,\cA(L))<\a_2\,\} \subset U$, where $U$ is from \eqref{Lambda0}.
This process can be repeated as long as $c(Y_n,\cA(L))<\a_2$ holds and \eqref{ALGA} is not satisfied.
The resulting periodic  $Y_n$ is in $\La$ and hence by \eqref{pl10} it has period larger than 10.
Thus the process stops at an iterate $N$ where the period in \eqref{peryn} is larger than 1.
 This is 
\begin{align}
N\le &\log_R\per(y_0) \le \log_R (5TP_T),
\notag\\
A_L(y_N)\le (B_4)^{2N} A_1(T)
&\le 
(5 T P_T)^{2\log_R B_4} \cdot K_5 \,P_T\,\ee^{-2\la T}
\qquad\text{using \eqref{A1T},}
\notag\\
c(Y_N,\cA(L)) &\le 
(5TP_T)^{\log_R B_4} \sqrt{K_5 P_T}\; \ee^{-\la T},
\notag\\
|v^0_k|+\tsum_{i=1}^n|v^i|
&\le 
(5TP_T)^{\log_R B_4} \sqrt{K_5 P_T}\; \ee^{-\la T}.
\label{tshift}
\end{align}
Since by~\eqref{Psubexp}, $P_T$ has sub-exponential growth in $T$,
we have that 
  $c(Y_N,\cA(L))$ and $A_L(y_N)$ can be made 
arbitrarily small by choosing $T$ sufficiently large.
Then the process stops not because $c(Y_N,\cA(L))$ is large,
but because  \eqref{ALGA} holds.

\end{proof}

\section{The perturbed minimizers.}
\label{SS4}

The following Crossing Lemma is extracted for Mather~\cite{Mat5} with 
the observation that the estimates can be taken uniformly on a $C^2$ neighbourhood of $L$.

\begin{Lemma}[Mather~{\cite[p. 186]{Mat5}}]\label{CL}\quad

If $K>0$, then there exist $\e$, $\de$, $\eta$, $\zeta>0$ and 
\begin{equation}\label{CL>1}
C> 1,
\end{equation}
such that 
if $\lV \phi\rV_{C^2}<\zeta$, and 
$\a, \ga:[t_0-\e,t_0+\e]\to M$
are solutions of the Euler-Lagrange equation for $L+\phi$
with $\lV d\a(t_0)\rV$, $\lV d\ga(t_0)\rV\le  K$,
$d\big(\a(t_0),\ga(t_0)\big)\le\de$, and
$$d\big(d\a(t_0),d\ga(t_0)\big)\ge C\; d\big(\a(t_0),\ga(t_0)\big),$$
then there exist $C^1$ curves $a,c:[t_0-\e,t_0+\e]\to M$
such that $a(t_0-\e)=\a(t_0-\e)$, $a(t_0+\e)=\ga(t_0+\e)$,
$c(t_0-\e)=\ga(t_0-\e)$, $c(t_0+\e)=\a(t_0+\e)$, and
\begin{equation}\label{ecl}
A_{L+\phi}(\a)+A_{L+\phi}(\ga)
-A_{L+\phi}(a)-A_{L+\phi}(c)
> \eta\; d\big(d\a(t_0),d\ga(t_0)\big)^2.
\end{equation}

\end{Lemma}

\medskip

\begin{Lemma}\label{LKTr}\quad

Given a Tonelli lagrangian  $L_0$ and a compact subset $\De\subset TM$,  
there are  $\e>0$, $K>0$ and $\de_1>0$ 
such that for any Tonelli lagrangian $L$ with $\lV (L-L_0)|_{B_\e(\De)}\rV_{C^2}<\e$,
\linebreak
$B_\e(\De):=\{\th\in TM: d(\th,\De)<\e\}$, and any $T>0$:
\begin{enumerate}[(a)]
\item\label{KTra}
If $x\in C^1([0,T],M)$  is a solution of the Euler-Lagrange equation for $L$
with $(x,\dx)\in \De$ and 
$z\in C^1([0,T],M)$ satisfies
$$
d\big([z(t),\dz(t)],[x(t),\dx(t)]\big) \le 4 \rho\le\de_1 \qquad \forall t\in[0,T],
$$
then 
\begin{equation}\label{KTr}
\lv \int_0^T L(z,\dz)\,dt -\int_0^T L(x,\dx)\, dt -\partial_v L(x,\dx)\cdot(z-x)\Big|_0^T\rv
\le K \,(1+T)\,  \rho^2,
\end{equation}
where $z-x:=(\exp_x)^{-1}(z)$.

\item\label{KTrb} 
If $x\in C^1([0,T],M)$ is a solution of the Euler-Lagrange equation for $L$ with $(x,\dx)\in \De$
 and
the curves 
$w_1,\,w_2,\,z\in C^1([0,T],M)$ satisfy
$w_1(0)=x(0)$, $w_1(T)=z(T)$, $w_2(0)=z(0)$, $w_2(T)=x(T)$, and 
for all $\xi\in\{z,\,w_1,\,w_2\}$ we have
$$
d\big([\xi(t),\dxi(t)],[x(t),\dx(t)]\big) \le 4 \rho\le \de_1 \qquad \forall t\in[0,T], 
$$
then
$$
\lv A_L(x)+A_L(z)-A_L(w_1)-A_L(w_2)\rv \le 3 K  \rho^2  (1+T).
$$
\end{enumerate}
\end{Lemma}

\noindent{\bf Proof:}
\begin{enumerate}[(a)]
\item
We use a coordinate system on a tubular neighbourhood of $x([0,T])$ with 
a bound in the  $C^2$ norm independent of $T$ and of $\dx(0)$. 
In case $x$ has self-intersections
or short returns the coordinate system is an immersion. 

We have that
\begin{align*}
L(z,\dz)-L(x,\dx) = \partial_x L(x,\dx)(z-x) +\partial_v L(x,\dx)(\dz-\dx) +O(\rho^2),
\end{align*}
here $O(\rho^2)\le K\, \rho^2$
where $K$ depends on the second derivatives of $L$ on a small neighbourhood of 
 the compact $\De$ and hence
it can be taken uniform on a $C^2$ neighbourhood of $L$. Since $x$ satisfies the 
Euler-Lagrange equation for $L$,
\begin{align*}
L(z,\dz)-L(x,\dx) = \tfrac d{dt}[\partial_v L(x,\dx)(z-x)] +O(\rho^2).
\end{align*}
This implies \eqref{KTr}.

\item By item~\eqref{KTra}
\begin{align*}
A_L(w_1) -A_L(x) &\le \partial_vL(x,\dx) (w_1-x)\Big|_0^T + K \rho^2 (1+T)
\\
&\le   \partial_vL(x(T),\dx(T)) (z(T)-x(T)) + K \rho^2 (1+T) .
\\
A_L(w_2)-A_L(x) 
&\le -\partial_vL(x(0),\dx(0)) (z(0)-x(0)) + K \rho^2 (1+T) .
\\
A_L(x)-A_L(z) 
&\le  -\partial_vL(x,\dx) (z-x)\Big|_0^T + K \rho^2 (1+T)
\\
&\le -\partial_vL(x(T),\dx(T)) (z(T)-x(T)) 
\\
&\hskip 10pt 
+\partial_vL(x(0),\dx(0)) (z(0)-x(0))+ K \rho^2 (1+T) .
\end{align*}
Adding these inequalities we get
$$
A_L(w_1)+A_L(w_2)-A_L(x)-A_L(z)
\le 3 K \rho^2 (1+T).
$$
The remaining inequality is obtained similarly.
\qed
\end{enumerate}

\openup -0.5pt
The following proposition has its origin in Yuan and Hunt \cite{YH}, 
the present proof uses some arguments by Quas and Siefken \cite{QS}.
Proposition~\ref{Ppert} together with  proposition~\ref{palga} imply theorem~\ref{HYP}.

\begin{Proposition}\label{Ppert}\quad

Suppose that for every $\de>0$ there is a periodic orbit $\Ga\subset \La\subset E_L^{-1}\{c(L)\}$ 
 such that 
\begin{equation}\label{ALGA2}
c(\Ga,\cA(L))<\de\,\ga(\Ga) 
\qquad\text{and}\qquad
A_{L+c(L)}(\Ga) < \de^2 \, \ga(\Ga)^2,
\end{equation}
where \qquad
$\ga(\Ga):=\min\{ d_{TM}(\Ga(s),\Ga(t)):|s-t|_{\rm mod(\per{\Ga})}\ge 1\,\}$.

Then for any $\e>0$ there is $\phi\in C^2(M,\re)$ with $\lV \phi\rV_{C^2}<\e$
such that $\Ga \subset \cA(L+\phi)$,
where $\Ga$ is one of the periodic orbits in \eqref{ALGA2}.
\end{Proposition}

{\bf Idea of the Proof:}

We choose $\de=\de(\e)$ sufficiently small and use the periodic orbit $\Ga$
given by the hypothesis.
We perturb the Lagrangian by a potential $\phi$ which is a 
non-negative channel centered at $\pi(\Ga)$
defined in~\eqref{defphi}. 
The curve $\Ga$ is a periodic orbit for the flows of $L$ and of $L+\phi$.
We show that $\Ga$ is contained in the Aubry set $\cA(L+\phi)$ by proving that
any semi-static curve $x:]-\infty,0]\to M$ for $L+\phi$ has 
$$
\a\text{-limit of }(x,\dx)=\Ga;
$$
because by Ma\~n\'e \cite[Theorem V.(c)]{Ma7}, $\a$-limits of semi-static orbits are static.
This is done by calculating the action of each segment of the semi-static which is spent 
outside of a small neighbourhood of $\Ga$, and proving that it has a uniform positive
lower bound. Since the total action of a semi-static is finite, the quantity of those segments 
is finite. Thus the semi-static eventually stays forever in a small neighbourhood of $\Ga$.
The expansivity of $\La(L+\phi)\supset\cN(L+\phi)$ implies that the $\a$-limit of the semi-static is $\Ga$.

\begin{Lemma}\label{Lgade}\quad

If $\cA(L)$ has no periodic orbits and $\Ga_n$ is a sequence of periodic orbits with
\begin{gather*}
c(\Ga_n,\cA(L))<\de_n \cdot \diam(\cA(L)),
\\
\ga_n:=\min\{d(\Ga_n(s),\Ga_n(t)) : |s-t|_{\rm{mod(per\, }\Ga_n)}\ge 1\,\}.
\end{gather*}
Then \quad
$\lim_n \de_n =0$ \quad  $\then$ \quad $\lim_n\ga_n=0$.
\end{Lemma}

\begin{proof}
Let $T_n$ be the period of $\Ga_n$.
First we prove that $\lim_n T_n=\infty$.
If not, we can extract a subsequence where $\th:=\lim_n \Ga_n(0)\in\cA(L)$ and $S:=\lim_n T_n$ exist.
Then $\th$ is a periodic point in $\cA(L)$ which contradicts the hypothesis.

Consider the points $\Ga_n(4m)$, $0\le m\le M_n:=[\tfrac 14 T_n]$, $m\in\na$.
Since $\lim_nT_n=\infty$, the quantity $M_n$ of these points tends to infinity.
Therefore 
$$
\ga_n \le \min_{m_1\ne m_2}d(\Ga_n(m_1),\Ga_n(m_2))
\overset{n}\longrightarrow 0.
$$
\end{proof}
\openup +0.5pt

\bigskip

{\bf Proof of Proposition~\ref{Ppert}:}

By adding a constant to $L$ we can assume that 
\begin{equation}\label{clo}
c(L)=0.
\end{equation}
Fix $K_1>0$ such that
\begin{equation}\label{dK1}
[E_L\le c(L)+1]\subset [|v|\le K_1].
\end{equation}

Bernard \cite{Be3} after Fathi and Siconolfi  \cite{FaSi} proves that there is 
 a $C^{1+\Lip}$ critical subsolution $u$ of the Hamilton-Jacobi equation 
for $L$, $H(x, d_xu)\le c(L)$. Thus
\begin{equation}\label{Ldu}
L - du \ge 0.
\end{equation} 

By Gronwall's inequality and the continuity of Ma\~n\'e's critical value $c(L)$
(see 
\cite[Lemma~5.1]{CP})  there is $\a>0$ and $\ga_0$  such that if 
$\lV\phi\rV_{C^2}\le 1$,  $0<\ga<\ga_0$ and
$\Ga$ is a periodic orbit for $L+\phi$ 
with energy smaller than $c(L+\phi)+1$ 
then
\begin{equation}\label{tima}
d(\vr^{L+\phi}_s(\vrt),\Ga)\le\frac\ga 4 
\quad \text{ and }\quad
d(\vr^{L+\phi}_t(\vrt),\Ga)\ge\frac\ga3
\quad\then \quad 
|t-s|> \a.
\end{equation}

The graph property states that the projection $\pi:\cA(L)\to M$ has a Lipschitz
inverse (see Ma\~n\'e \cite{Ma7}). The Lipschitz constant is the same as $C$ in
Mather's Crossing Lemma~\ref{CL}.
The Aubry set has energy $c(L)$ and $c(L+\phi)$ is continuous on $\phi$.
Then one can choose 
\begin{equation}\label{e1ze}
\e_1<\zeta
\end{equation}
 and $K$, $C>1$ in Lemma~\ref{CL}  such that if
$\lV \phi\rV_{C^2}<\e_1$ then $\cA(L+\phi)$ is a graph with Lipschitz constant $C$.

By the upper semicontinuity of the Ma\~n\'e set \cite[lemma 5.2]{CP} we can 
choose a neighbourhood $U$ of $\cN(L)$
and $0<\e_2<\e_1$  such that if $\lV \phi\rV_{C^2}<\e_2$
then the set 
$$
\La(\phi) := \bigcap_{t\in\re}\vr_{-t}^{L+\phi}(\ov{U})
$$
is hyperbolic and contains $\cN(L+\phi)$.
Take $0<\e_3<\e_2$ such that $\La(\phi)$ has uniform constants of 
hyperbolicity (\ref{B71}), expansivity (\ref{dfe}, \ref{rue}) and canonical coordinates (\ref{caco})
for all $\lV\phi\rV_{C^2}<\e_3<\e_2$.

Write
\begin{equation}\label{54gad}
\gad:=\ga(\Ga).
\end{equation}
We can assume that $\cA(L)$ has no periodic points.
By lemma~\ref{Lgade}, $\ga_\de$ is small when $\de$ is small.
Given $0<\e<\e_3$, choose $0<\de\ll \e$ and a periodic orbit $\Ga$
satisfying~\eqref{ALGA2}
with $\de$
and     $\ga_\de$ so small that
for all $\lV\phi\rV_{C^2}<\e_3$,
\begin{align}
\gad &< \epsilon_0  \hskip 6pt \text{where $\epsilon_0$ is a  flow expansivity constant for $\La(\phi)$ 
 as in
\ref{dfe} and \ref{rue}. }
\label{gadeo1}
\\
2 \gad &< \de_1 \hskip 5pt \text{with }\de_1:=\de[K_1]
\text{  from lemma~\ref{CL}, where  $K_1$ is from~\eqref{dK1}.
}
\label{dfde1}
\\
\gad &< \be_0 \hskip 4pt
\text{where  $\be_0$  is from proposition~\ref{B71} for $\La(\phi)$.}
\label{gadbe0}
\\
\gad &<  \eta_0
\hskip 9pt
 \text{where $\eta_0$ is from the canonical coordinates in \ref{B4} for $\La(\phi)$,}
\label{gadeta}
\end{align}
and such that writing 
\begin{equation}\label{dgad}
\ogd:=\frac{\gad}{3C(B+1)} <\tfrac 12 \,\gad,
\end{equation}
we have that 
\begin{align}
&\ogd<\ga_0\hskip 2cm\text{ where $\ga_0$ is from~\eqref{tima},}
\label{gatima}
\end{align}

\noindent and there is $\rho$,
\begin{equation}\label{rho14ga1}
 \de\,\gad<\rho<  \tfrac 14 \ogd\ll 1
\end{equation} 
such that
\begin{gather}
\tfrac 1{4}\,\e\, \rho^{2} >  \de^2\,(\gad)^2,
\label{rho11}
\\
 C\rho > \tfrac1{\sqrt{\eta_1}}\,{\de\,\ga_\de},
 \label{rho151}
 \\
\big( \tfrac 1{32}\, \e\, (\ogd)^2 - \de^2(\gad)^2\big) \a
- 6 K D^2  C^2 (B+1)^2\rho^2
-3\, \de^2(\gad)^2
>0,
\label{rho2}
 \end{gather}
where $B$ is from Lemma~\ref{B4}, 
$C=C[K_1]$ and $\eta_1=\eta[K_1]$ are from Lemma~\ref{CL} with
\begin{equation}\label{C>1}
C>1,
\end{equation}
 $D$ is from Proposition~\ref{B71} 
and $K$ is from Lemma~\ref{LKTr}
applied to the compact
\linebreak
$\Delta=
[E_L\le c(L)+5]$.
Inequality \eqref{rho2} implies 
\begin{equation}\label{rho3}
 \tfrac 1{32}\, \e\, (\ogd)^2 > \de^2(\gad)^2.
\end{equation}

Let $\phi:M\to[0,1]$ be a $C^\infty$ function such that  $\lV \phi\rV_{C^{2}}<10\, \e$ and
\begin{equation}\label{defphi}
0\le \phi(x)=
\begin{cases}
0 &\text{if } x\in\pi(\Ga),
\\
\ge \tfrac 14\, \e \, \rho^{2} 
&\text{if } d(x,\pi\Ga)\ge \rho,
\\
\tfrac 1{ 32 }\,\e\, (\ogd)^{2}
&\text{if }d(x,\pi\Ga)\ge \tfrac 14 \ogd.
\end{cases}
\end{equation}

Using $u$ from \eqref{Ldu} write
\begin{equation}\label{defLL}
\L:=L+\phi+c(L+\phi)-du.
\end{equation}
The Euler-Lagrange flow of $\L$ and the sets $\cA(\L)$, $\cN(\L)$ are the same as those of
$L+\phi$. In particular the hyperbolicity constants \eqref{gadeo1}--\eqref{gadeta} and 
Lipschitz graphs contants \eqref{C>1} remain valid for $\L$.
\medskip

\pagebreak

\begin{Claim}
{ If $\de$ is small enough then}
\begin{enumerate}
\item \label{cl1}
We have that
$$
\inf\limits_{d(s,t)_{\text{mod }T}\ge 1} d\big(\pi\Ga(s),\pi\Ga(t)\big)
 >\tfrac 34 \,\ogd.
 $$
 In particular the neighbourhood $B(\pi\Ga,\tfrac 38\ogd)$ of $\pi\Ga$ of radius 
 $\frac 38 \ogd$ has no self intersections, i.e. it is homeomorphic to
 $S^1\times ]0,1[^{\dim M-1}$.
 
\item\label{cl2}
 If  $x:]-\infty,0]\to M$ is a semi-static orbit for $\L$ then for all $t\le -1$
\begin{align}\label{ecl21}
&\text{either }\quad  
d\big([x(t),\dx(t)],\Ga\big) \le \tfrac{ \de\, \ga(\Ga)}{\sqrt{\eta_1}}
\quad\text{ or } \quad d\big([x(t),\dx(t)],\Ga\big) \le C\, d\big(x(t),\pi\Ga\big),
\\
&\text{or }\quad d(x(t),\pi\Ga)\ge \de_1,
\label{ecl31}
\end{align}
where   $\eta_1=\eta(K_1)$, $C=C(K_1)$
and $\de_1=\de_1(K_1)$ are from Lemma~\ref{CL} for $K=K_1$ from~\eqref{dK1}.
\end{enumerate}
\end{Claim}

{\it Proof:}

Let $T=\per(\Ga)$ be the period of $\Ga$.

\eqref{cl1}. Given $s, t\in[0,T]$, by~\eqref{ALGA2} there are $\tt_s,\tt_t\in\cA(L)$ such that
\begin{align*}
d(\pi\Ga(s),\pi\tt_s)&\le d(\Ga(s),\tt_s) < \de\, \ga(\Ga),
\\
d(\pi\Ga(t),\pi\tt_t)&\le d(\Ga(t),\tt_t) < \de\, \ga(\Ga).
\end{align*}
If $d(s,t)_{\text{\rm mod }T}\ge 1$  then 
\begin{align*}
d(\tt_s,\tt_t)&\ge d(\Ga(s),\Ga(t))-d(\Ga(s),\tt_s)-d(\Ga(t),\tt_t)
\\
&> \ga(\Ga)-2\de\, \ga(\Ga).
\end{align*}
Since $\tt_s,\tt_t\in \cA(L)$, by the graph property~\ref{CL} for $\cA(L)$ 
and \eqref{dgad}, \eqref{CL>1} we have that 
$$
d(\pi\tt_s,\pi\tt_t)\ge \tfrac 1C\,d(\tt_s,\tt_t)
\ge \frac {\ga(\Ga)(1-2\de)}C > \ogd -2\de\, \gad.
$$

 Then
 \begin{align*}
 d(\pi\Ga(s),\pi\Ga(t)) &\ge d(\pi\tt_s,\pi\tt_t)-d(\pi\Ga(s),\pi \tt_s)-d(\pi\Ga(t),\pi \tt_t) \\
 &>\ogd-4\de\, \gad> \tfrac 34\, \ogd.
 \end{align*}

\eqref{cl2}. Suppose by contradiction that there exists $t\le -1$ such that 
\begin{gather}\label{posso1}
d(x(t),\pi\Ga)<\de_1 \qquad \text{ and }
\\
 d\big([x(t),\dx(t)],\Ga\big)^2 > \frac{\de^2\, \ga(\Ga)^2}{\eta_1} 
\quad\text{ and } \quad d\big([x(t),\dx(t)],\Ga\big) > C\, d\big(x(t),\pi\Ga\big).
\label{posso3}
\end{gather}

First we check that we can apply the Crossing Lemma~\ref{CL} to $\L$.
Given  $\ga:[0,S]\to M$ we have that 
$$
\oint_\ga c(L+\phi)-du = S \, c(L+\phi) -u(\ga(S))+u(\ga(0))
$$
depends only on the time interval $S$ and the endpoints of $\ga$.
Thus instead of  $\L$ in~\eqref{defLL}, it is enough to apply Lemma~\ref{CL} to $L+\phi$, for whom 
it holds if $\lV \phi\rV_{C^2}<\e_1<\zeta$ by \eqref{e1ze}.

Now we check the speed hypothesis in Lemma~\ref{CL}.
Observe that 
$$
E_\L=v\,\L_v-\L=E_{L+\phi}-c(L+\phi)=E_L-\phi-c(L+\phi),
$$
 and that by~\eqref{minmeas}
 $$
 c(\L)=c\big(L+\phi+c(L+\phi)\big)=0.
 $$
 Therefore
 $$
 \cN(\L)\subset [E_\L= c(\L)]
 \subset [E_L=\phi+c(L+\phi)].
 $$
 If $\phi$ is small enough 
$$
\phi+c(L+\phi)<c(L)+1,
$$
and then $\dx(t)\in
\cN(\L)\subset [E_L\le c(L)+1]$. 
By hypothesis in \ref{Ppert},
$\Ga\subset [E_L=c(L)]$.
Therefore by~\eqref{dK1},
$$
\forall t \qquad
\dx(t),\; \Ga(t) \in [E_L\le c(L)+1]\subset [|v|\le K_1].
$$

Finally we check the distance hypothesis in Lemma~\ref{CL}. 
Let $t_0$ be such that $d(x(t),\pi\Ga)=d(x(t),\pi(\Ga(t_0)))$.
By~\eqref{posso1} and the definition of $\de_1$ in \eqref{dfde1} we can apply Lemma~\ref{CL} for $\L$ 
and  $K=K_1$ from~\eqref{dK1},
to $x$ and $\pi\Ga$ at $x(t)$ and $\pi(\Ga(t_0))$.
Also note that by~\eqref{posso3} we have that, as required in Lemma~\ref{CL},
$$
d\big([x(t),\dx(t)],\Ga(t_0)\big)\ge d\big([x(t),\dx(t)],\Ga\big) > C\, d\big(x(t),\pi\Ga\big)=C\, d\big(x(t),\pi\Ga(t_0)\big).
$$

Using $0<\e\le 1$ from Lemma~\ref{CL} we obtain $C^1$ curves
$w_1,\,w_2:[-\e,\e]\to M$ with $w_1(-\e)=x(t-\e)$, $w_1(\e)=\pi\Ga(t_0+\e)$,
$w_2(-\e)=\pi\Ga(t_0-\e)$, $w_2(\e)=x(t+\e)$ such that
$$
A_\L(w_1)+A_\L(w_2) 
< A_\L(\pi\Ga|_{[t_0-\e,t_0+\e]}) + A_\L(x|_{[t-\e,t+\e]})
- \eta_1\, d([x(t),\dx(t)],\Ga(t_0))^2.
$$

Since $\phi\ge 0$ and \eqref{clo} we have that
\begin{equation}\label{clf}
c(L+\phi)\le c(L)=0.
\end{equation}
Using 
\eqref{ALGA2}, $\phi|_{\pi\Ga}\equiv 0$ and that $\pi\Ga$ is a closed curve
we have that 
$$
A_\L(\pi\Ga) = A_{L+c(L+\phi)}(\pi\Ga) \le A_{L+c(L)}(\pi\Ga) <
\de^2\,\ga(\Ga)^2.
$$
We compute the action of the curve $w_1*\pi\Ga|_{[t_0+\e,t_0+T-\e]}*w_2$
which joins $x(t-\e)$ to $x(t+\e)$.
\begin{align*}
A_\L(&w_1)+A_\L(\pi\Ga|_{[t_0+\e,t_0+T-\e]})+A_\L(w_2) <
\\
&< A_\L(x|_{[t-\e,t+\e]})
+ A_\L(\pi\Ga|_{[t_0-\e,t_0+\e]}) + \A_\L(\pi\Ga|_{[t_0+\e,t_0+T-\e]})
- \eta_1\, d([x(t),\dx(t)],\Ga(t_0))^2
\\
&<  A_\L(x|_{[t-\e,t+\e]})
+ \de^2 \,\ga(\Ga)^2 - \eta_1\, d([x(t),\dx(t)],\Ga)^2
\\
&< A_\L(x|_{[t-\e,t+\e]}), \qquad \text{using~\eqref{posso3}.}
\end{align*}
This contradicts the assumption that $x$ is semi-static for $\L$.

\hfill$\triangle$

Since we can assume that $\cA(L)$
has no periodic orbits, if $\de$ is small enough 
\begin{equation}\label{T>1}
T:=\per(\Ga)> 1.
\end{equation}

Observe that $\Ga$ is also a periodic orbit for $L+\phi$.
Let $\mu_\Ga$ be the invariant probability supported on $\Ga$.
Using \eqref{minmeas},  \eqref{clo}, \eqref{ALGA2} we have that 
\begin{align}
c(L+\phi) &\ge -\int(L+\phi)\,d\mu_\Ga
=-\int L\;d\mu_\Ga
\notag
\\
&\ge - \tfrac 1T \,\de^2\,\ga(\Ga)^2. 
\label{cLp1}
\end{align}

\bigskip

We will prove that  any semi-static curve $x:]-\infty,0]\to M$ for $L+\phi$
has 
$\a$-limit$\{(x,\dx)\}$ $= \Ga$.
Since $\a$-limits of semi-static orbits are static (Ma\~n\'e \cite[Theorem V.(c)]{Ma7}),
this implies that $\Ga\subset\cA(L+\phi)$.
Thus finishing the proof of Proposition~\ref{Ppert}.

Since by~\eqref{gadeo1}, 
the number
 $\ogd$ is smaller than the flow expansivity constant of $\cN(L+\phi)$,
it is enough to prove that  the tangent $(x,\dx)$ of 
any semi-static curve $x:]-\infty,0]\to M$ 
spends only  a bounded time outside the $\tfrac 38\ogd$-neighbourhood of $\Ga$.

Let $x:]-\infty,0]\to M$ be a semi-static curve for $L+\phi$.
Let $\tt:=(x(0),\dx(0))$ and let $\psi_t=\vr_t^{L+\phi}$ be the lagrangian flow of $L+\phi$.
By~\eqref{dfde1} and~\eqref{dgad} we have that
\begin{equation}\label{dxpgd2}
d(x(t),\pi\Ga)\ge \de_1
\quad\then\quad
d(x(t),\pi\Ga)>\tfrac 14 \ogd.
\end{equation}

By \eqref{ecl21}-\eqref{ecl31} and \eqref{rho151} we have that
\begin{equation}\label{Drho1}
d(\psi_t(\tt),\Ga)> C\rho
\quad \&\quad 
d(x(t),\pi\Ga)<\de_1
 \quad \then \quad 
d(x(t),\pi\Ga) \ge \tfrac 1C \, d(\psi_t(\tt),\Ga).
\end{equation}
By~\eqref{rho14ga1} and~\eqref{dgad} we have that $\tfrac 14\ga_\de>C\rho$.
And then  from~\eqref{dxpgd2} and~\eqref{Drho1} we get
\begin{equation}\label{Crho}
d(\psi_t(\tt),\Ga)\ge\tfrac 14 \gad\quad\left(> \tfrac 14 C\,\ogd\right)
 \quad \then \quad 
d(x(t),\pi\Ga) 
> \tfrac 14\,\ogd.
\end{equation}

Also, from~\eqref{dxpgd2},~\eqref{Drho1}  and \eqref{rho14ga1} we have that
\begin{equation}\label{dpsgacrho}
d(\psi_t(\tt),\Ga)>C\rho \quad \then \quad
d(x(t),\pi\Ga) > \rho.
\end{equation}
Then by \eqref{dpsgacrho},  \eqref{defphi},  \eqref{cLp1},  \eqref{T>1} and \eqref{rho11}, 
 we have that 
\begin{equation}\label{phirho}
d(\psi_t(\tt),\Ga)>C\rho \quad \then \quad
\phi(x(t))+c(L+\phi) \ge 
 \tfrac 1{4}\,\e \rho^{2} -  \de^2 \,\ga_\de^2 =:a_0>0.
\end{equation}

For $\xi\in\La(\phi)$ consider the local invariant manifolds
\begin{align*}
W^s_\eta(\xi)&:=\{\,\zeta\in E_\L^{-1}\{c(\L)\}\;:\; \forall t\ge0 \quad
d(\psi_t(\zeta),\psi_t(\xi))\le \eta\,\},
\\
W^{ss}_\eta(\xi) &:=\{\, \zeta\in W^s_\eta(\xi)\;:\; \lim_{t\to+\infty}
d(\psi_t(\zeta),\psi_t(\xi))=0\,\},
\\
W^u_\eta(\xi)&:=\{\,\zeta\in E_{\L}^{-1}\{c(\L)\}\;:\; \forall t\le0 \quad
d(\psi_t(\zeta),\psi_t(\xi))\le \eta\,\},
\\
W^{uu}_\eta(\xi) &:=\{\, \zeta\in W^u_\eta(\xi)\;:\; \lim_{t\to-\infty}
d(\psi_t(\zeta),\psi_t(\xi))=0\,\}.
\end{align*}

Also consider the canonical coordinates as in~\ref{B4} on $\La(\phi)$, i.e.
{\sl there are $\eta_0,\,\eta>0$ such that if
 $\xi,\,\zeta\in\La(\phi)$
and $d(\xi,\zeta)<\eta_0$
then there is $v=v(\xi,\zeta)\in\re$, $|v|\le\eta$ such that
\begin{align}
\langle \xi, \zeta \rangle:=W^{ss}_\eta(\psi_v(\xi))\cap W^{uu}_\eta(\zeta)
\ne \emptyset.
\label{cacoN}
\end{align}
}
We use the canonical coordinates to parametrize the approaches of
$\psi_t(\tt)$ to $\Ga$ in the following way. By \eqref{gadeta}, $\gad<\eta_0$.
The local weak stable manifold of $\Ga$
$$
W^{s}_\eta(\Ga):=\textstyle\bigcup_{\xi\in\Ga}W^{s}_\eta(\xi)
=\bigcup_{\xi\in\Ga}W^{ss}_\eta(\xi)
$$
forms a cylinder homeomorphic to $\Ga(\re)\times ]0,1[^{\dim M-1}$.
When $d(\psi_t(\tt),\Ga(\re))<\gad$ the strong local unstable manifold
$W^{uu}_\eta(\psi_t(\tt))$ intersects this cylinder transversely  and 
defines a unique time parameter $v(t)$ (mod $T$) such that 
\begin{equation}\label{vttt}
W^{ss}_\eta(\Ga(v(t)))\cap W^{uu}_\eta(\psi_t(\tt))\ne 0.
\end{equation}
Since the family of strong invariant manifolds is invariant 
under each iterate $\psi_t$
we have that if $d(\psi_t(\tt),\Ga(\re))<\gad$ for all $t\in[a,b]$
then
$$
\forall s\in[0,b-a]\qquad v(a+s)= v(a)+s.
$$

 Let  $B$ be from Lemma~\ref{B4}.
 Write
$\tt=(x(0),\dx(0))$ and
define $S_k(\tt)$, $T_k(\tt)$ recursively by
\begin{align}
S_0(\tt)&:=0,
\label{defSTCS}\\
T_k(\tt)&:=\sup\,\big\{\,t< S_{k-1}(\tt)\;\big|\; d\big(\psi_t(\tt),\Ga(v(t))\big)
\le C(B+1)\rho\,\big\},
\notag
\\
C_k(\tt)&:=\sup \,\big\{\,t<T_k(\tt)\;\big|\; d(\psi_t(\tt),\Ga(\re))=\tfrac13\gad\,\big\},
\notag
\\
S_k(\tt)&:=\inf\big\{\,t>C_k(\tt)\;\big|\;d\big(\psi_t(\tt),\Ga(v(t))\big)\le C(B+1)\rho\,\big\}.
\notag
\end{align}

\begin{figure}[h]
\resizebox*{13cm}{3.5cm}{\includegraphics{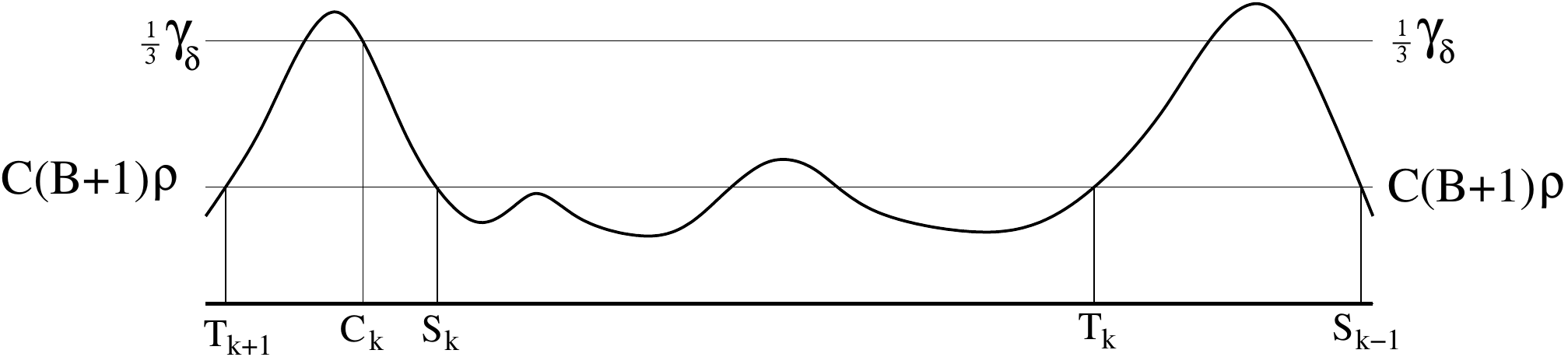}} 
\caption{This figure illustrates the distance of the orbit of $\tt$ to the
periodic orbit $\Ga$ and the choice of $S_k$, $T_k$ and $C_k$.
}\label{543escapes}
\end{figure}

\pagebreak

\begin{Claim}\label{stcs}\quad

\begin{enumerate}
\item\label{stcs1} If $S_{k-1}(\tt)>-\infty$ then $T_k(\tt)>-\infty$.
\item\label{stcs25} If $T_k(\tt)>-\infty$ then $T_{k+1}(\tt) \le C_k(\tt)$.
\item\label{stcs15} If $C_{k-1}(\tt)>-\infty$ then
$d\big[\psi_{T_k(\tt)}(\tt),\Ga(v(T_k(\tt)))\big]=C(B+1)\rho$.
\item\label{stcs2} If $C_k(\tt)>-\infty$ then $C_k(\tt)<S_k(\tt)\le T_k(\tt)$.
\item\label{stcs3} If the sequence $\{T_k\}$ is finite, then 
$\a\text{-limit}(x,\dx)=\Ga$.
\item\label{stcs6} If $t\in[S_k(\tt),T_k(\tt)]$ then $d(\psi_t (\tt),\Ga(\re))\le\tfrac 13\gad$.
\end{enumerate}

\end{Claim}

\noindent{\it Proof:}

\noindent\eqref{stcs1}. 
 Suppose by contradiction that $S_{k-1}(\tt)>-\infty$ 
 but $T_k(\tt)=-\infty$.
Let $\Phi^L_k$ be the action potential \eqref{actionpotential} for $L$. 
Since $\Phi^L_{c(L)}$ is Lipschitz, it is bounded 
on $M\times M$.
\begin{align}
\int_{-t}^{S_{k-1}(\tt)}L(x,\dx)
=
\int_{-t}^{S_{k-1}(\tt)}\big\{c(L)+L(x,\dx)\big\} 
&\ge \Phi^L_{c(L)}\big(x(-t),x(S_{k-1}(\tt))\big)
\notag
\\
&\ge\inf_{y,z\in M}\Phi^L_{c(L)}(y,z)=:b_0 >-\infty.
\label{intb0}
\end{align}
Recall that $\eta$ is from the canonical coordinates~\ref{caco} for $\La(\phi)$ 
as in~\eqref{cacoN}
and satisfies~\eqref{gadeta}.
Since $T_k(\tt)=-\infty$ we have that for all $t<S_{k-1}(\tt)$ either
\begin{equation}\label{caseta}
d(\psi_t(\tt),\Ga(\re))>\eta>\gad
\qquad \text{ or }
\end{equation}  
\begin{equation}\label{caseta1}
d(\psi_t(\tt),\Ga(\re))\le \eta\quad \text{ but }\quad
d(\psi_t(\tt),\Ga(v(t)))> C(B+1)\rho.
\end{equation}
In the case~\eqref{caseta1}
let $s(t)$ be such that $d(\psi_t(\tt),\Ga(s(t)))=d(\psi_t(\tt),\Ga(\re))\le\eta$.
We have that
\begin{align}
\langle \Ga(s(t)),\psi_t(\tt)\rangle
&=W^{s}_{\eta}(\Ga(s(t)))\cap W^{uu}_\eta(\psi_t(\tt))
\notag
\\
&=W^s_{\eta}(\Ga(v(t)))\cap W^{uu}_\eta(\psi_t(\tt))
=\langle \Ga(v(t)),\psi_t(\tt)\rangle
\notag
\\
&=W^{ss}_\eta(\Ga(v(t)))\cap W^{uu}_\eta(\psi_t(\tt)).
\label{Gavt}
\end{align}
 We apply Lemma~\ref{B4} with $x:=\Ga(s(t))$ and $y:=\psi_t(\tt)$.
 Using~\eqref{Bdxy} we have that
 \begin{equation}\label{dyvrv}
 d(y,\psi_v(x))\le d(y,x)+ d(x,\psi_v(x))
\le (1+B)\,d(y,x).
\end{equation}
Observe that \eqref{Gavt} implies that $\psi_v(x)=\Ga(v(t))$.
Replacing $x$ and $y$ in~\eqref{dyvrv} 
and using~\eqref{caseta1} we have that
\begin{align}
d(\psi_t(\tt),\Ga(\re))
&= d(\psi_t(\tt),\Ga(s(t)))
\ge \tfrac 1{1+B} \;d(\psi_t(\tt),\Ga(v(t)))
\notag
\\
&> C \rho.
\label{caseta3}
\end{align}
Observe that by~\eqref{rho14ga1} and~\eqref{dgad}, 
in case~\eqref{caseta} inequality~\eqref{caseta3} 
also holds. Therefore
\begin{equation}\label{tsk}
\forall t<S_{k-1}(\tt) \qquad
d(\psi_t(\tt),\Ga(\re)) > C \rho.
\end{equation}

Since $x$ is semi-static for $L+\phi$ we have for all $-t< S_{k-1}(\tt)$ that 
\begin{align}
\infty>
\sup_{y,z\in M}\Phi^{L+\phi}_{c(L+\phi)}(y,z)
&\ge
\Phi_{c(L+\phi)}^{L+\phi}\big(x(-t),x(S_{k-1}(\tt))\big)
\notag\\
&=\int_{-t}^{S_{k-1}(\tt)} \big[ L(x,\dx)+\phi(x)+c(L+\phi)\big]
\notag\\
&=\int_{-t}^{S_{k-1}(\tt)} L(x,\dx) 
+
\int_{-t}^{S_{k-1}(\tt)} 
\big[ \phi(x)+c(L+\phi) \big]
\notag\\
&\ge b_0 + a_0 \big(t+S_{k-1}(\tt)\big)
\qquad\text{ by \eqref{intb0} and \eqref{tsk}, \eqref{phirho}.}
\label{b+at}
\end{align}
By~\eqref{phirho} we have that $a_0>0$.
Letting $t\to+\infty$, inequality~\eqref{b+at} gives a contradiction.

\medskip
\noindent
\eqref{stcs25}.  
Let 
\begin{equation}\label{fgtv}
f(t):=d(\psi_t(\tt),\Ga(\re))
\qquad\text{  and }\qquad
g(t):=d(\psi_t(\tt),\Ga(v(t))),
\end{equation}
when $g$ is defined (in particular by~\eqref{gadeta} when $f(t)<\gad$).
Then $f(t)\le g(t)$.

Suppose first that $C_k(\tt)=-\infty$. 
Then $f(t)\ne \tfrac 13\gad$ for all $t<T_k(\tt)$.
By hypothesis  $T_k(\tt)>-\infty$, then 
$f(T_k(\tt))\le g(T_k(\tt))\le C(B+1)\rho$.
By~\eqref{rho14ga1},  $C(B+1)\rho<\tfrac 13\ga_\de$, and hence
 $f(t)<\tfrac 13\gad$ for all $t<T_k(\tt)$.
By~\eqref{gadbe0} and  Proposition~\ref{B71} with $L\to\infty$ we have that 
$\lim_{t\to-\infty}g(t)=0$.
Then $S_k(\tt)=-\infty$ and also $T_{k+1}(\tt)=-\infty$.

Now suppose that $C_k(\tt)>-\infty$.
By the definition of $S_k(\tt)$ 
for all $t\in]C_k(\tt),S_k(\tt)[$
we have that 
$g(t)> C(B+1)\rho$.
 This implies that $T_{k+1}(\tt)\le C_k(\tt)$.

\medskip
\noindent\eqref{stcs15}.
Let $f,\,g$ be as in~\eqref{fgtv}.
By the hypothesis $C_{k-1}(\tt)>-\infty$ and by the  definition of $C_{k-1}(\tt)$,
 $C_{k-1}(\tt)\le T_{k-1}(\tt)$. Then $f(C_{k-1}(\tt))=\tfrac 13\gad$.
By~\eqref{rho14ga1}, $C(B+1)\rho< \tfrac 13\gad$ and then
\begin{equation}\label{gcgfg}
C(B+1)\rho<\tfrac 13\gad=f(C_{k-1}(\tt))\le g(C_{k-1}(\tt)).
\end{equation}
By the definition of $S_{k-1}(\tt)$ we have that
$C_{k-1}(\tt)\le S_{k-1}(\tt)$.
But by~\eqref{gcgfg}, $g(C_{k-1}(\tt))\ge\tfrac 13\gad$,
and by the definition of  $S_{k-1}(\tt)$, if $S_{k-1}(\tt)<+\infty$
then $g(S_{k-1}(\tt))\le C(B+1)\rho<\tfrac 13\gad$.
Therefore $C_{k-1}(\tt)\ne S_{k-1}(\tt)$ and then 
\begin{equation}\label{cksk}
C_{k-1}(\tt)<S_{k-1}(\tt)\le+\infty.
\end{equation}
By~\eqref{gcgfg}
and the definition of $S_{k-1}(\tt)$ we have that 
$$
\forall t\in  ]C_{k-1}(\tt),S_{k-1}(\tt)[
\qquad
g(t)> C(B+1)\rho.
$$
This implies that $T_k(\tt)< C_{k-1}(\tt)$,
with strict inequality by~\eqref{gcgfg}.
By~\eqref{cksk} and  item~\eqref{stcs1} we have that $C_{k-1}(\tt)>-\infty$ 
implies that $T_k(\tt)>-\infty$.
Therefore
\begin{equation}\label{tkck1}
-\infty<T_k(\tt)<C_{k-1}(\tt)<S_{k-1}(\tt).
\end{equation}
The definition of $T_k(\tt)$ and the continuity of $g(t)$
on its domain
 imply that 
 \begin{equation}\label{gcb1}
 g(T_k(\tt))\le C(B+1)\rho.
 \end{equation}
 The domain of definition and continuity of $g$ contains
 $f^{-1}(]0,\gad[)\supset g^{-1}(]0,\gad[)$.
By the intermediate value theorem for $g$ on connected components 
of $[g\le\gad]$
and~\eqref{tkck1}, \eqref{gcb1}, \eqref{gcgfg}, the image 
$g([T_k(\tt),C_{k-1}(\tt)])$, and hence also 
$g(] -\infty,S_{k-1}(\tt)[)$,
contain the closed interval 
$\big[C(B+1)\rho,\tfrac 13\gad\big]$.
Therefore, by the definition of $T_k(\tt)$, we have that 
$g(T_k(\tt))=C(B+1)\rho$.

\medskip
\noindent\eqref{stcs2}.
Let $f$, $g$ be from~\eqref{fgtv}.
 If $C_k(\tt)>-\infty$ then 
 by the definition of $C_k(\tt)$,
 \begin{equation}\label{ckletk}
 C_k(\tt)\le T_k(\tt).
 \end{equation}
 Therefore $T_k(\tt)>-\infty$.
 Then the definition of $T_k(\tt)$ implies that 
 \begin{equation}\label{gtkcbr}
 g(T_k(\tt))\le C(B+1)\rho.
 \end{equation}
 Since $f(t)$ is continuous, 
 \begin{equation}\label{fck3g}
 f(C_k(\tt))=\tfrac 13\gad.
 \end{equation}
By~\eqref{gtkcbr}, ~\eqref{rho14ga1} and~\eqref{fck3g} we have that 
\begin{equation}\label{54c43gf}
g(T_k(\tt))\le C(B+1)\rho<\tfrac 14\ga_\de<\tfrac13\gad=f(C_k(\tt))\le g(C_k(\tt)).
\end{equation}
This implies that $C_k(\tt)\ne T_k(\tt)$.
This together with \eqref{ckletk} imply that
\begin{equation}\label{ck<tk}
C_k(\tt)<T_k(\tt).
\end{equation}
By~\eqref{gtkcbr} and~\eqref{ck<tk} 
the value $S_k(\tt)$ is an infimum of a set which contains $T_{k}(\tt)$,
therefore 
\begin{equation}\label{skztk}
S_k(\tt)\le T_k(\tt).
\end{equation}
This proves the second inequality in item~\eqref{stcs2}.

The first of the following inequalities follows from the definition of $S_k(\tt)$.
The second inequality is~\eqref{skztk}. The third inequality follows from
the definition of $T_k(\tt)$.
\begin{equation}\label{csts}
C_k(\tt)\le S_k(\tt)\le T_k(\tt)\le S_{k-1}(\tt).
\end{equation}
We get that
$$
-\infty <C_k(\tt)\le S_k(\tt)\le S_{k-1}(\tt)\le\cdots \le S_0(\tt):=0<+\infty.
$$
From the definition of $S_k(\tt)$ and $S_k(\tt)<+\infty$, and then~\eqref{fck3g},
 we have that
$$
g(S_k(\tt))\le C(B+1)\rho<\tfrac 13\gad=f(C_k(\tt))\le g(C_k(\tt)).
$$
In particular $C_k(\tt)\ne S_k(\tt)$.
Thus from~\eqref{csts}, $C_k(\tt)<S_k(\tt)$.

\medskip
\noindent
\eqref{stcs3}. If the sequence $\{T_k\}$ is finite, there is $\ell\in\na$ such that 
$T_\ell>-\infty$ and $T_{\ell+1}=-\infty$. Let $f,\,g$ be from~\eqref{fgtv}.
By item~\eqref{stcs25} we have that 
$-\infty< T_\ell(\tt)\le C_{\ell-1}(\tt)$.
Then we can apply item~\eqref{stcs15} and use~\eqref{rho14ga1}
to obtain
\begin{equation}\label{ftltt}
f(T_\ell(\tt))\le g(T_\ell(\tt))=C(B+1)\rho<\tfrac 13\gad.
\end{equation}
Since $T_{\ell+1}(\tt)=-\infty$, by item~\eqref{stcs1}, $S_{\ell}(\tt)=-\infty$
and by item~\eqref{stcs2}, $C_{\ell}(\tt)=-\infty$. 
Since $C_{\ell}(\tt)=-\infty$  
we have that $f(t)\ne \tfrac 13\gad$ for all $t<T_\ell(\tt)$.
But by~\eqref{ftltt}, $f(T_\ell(\tt))<\tfrac 13\gad$.
Since $f(t)$ is continuous, using~\eqref{gadbe0}  we get that
$$
f(t)<\tfrac 13\gad<\be_0 \qquad\text{ for all } t<T_\ell(\tt). 
$$
This implies that there is a continuous function $s:]-\infty,T_\ell(\tt)]\to\re$ such that
$$
\forall t\le T_k(\tt) \qquad d\big(\psi_t(\tt),\Ga(s(t))\big)\le \be_0.
$$
By Proposition~\ref{B71} there is $v\in\re$ and $\la>0$ such that 
$$
\forall t\le T_\ell(\tt)\qquad 
d(\psi_t(\tt),\Ga(t+v))\le D\,\be_0 \,\ee^{-\la (T_\ell(\tt)-t)}.
$$
This implies that $\lim\limits_{t\to +\infty}d(\psi_{-t}(\tt),\Ga)=0$ and that 
$\a\text{-limit}(\tt)=\Ga(\re)$.

\noindent\eqref{stcs6}.
By item~\eqref{stcs25}, $C_{k-1}(\tt)\ge T_k(\tt)>-\infty$.
By item~\eqref{stcs15} we have that  $f(T_k(\tt))\le g(T_k(\tt))=C(B+1)\rho<\tfrac 13\gad$.
By the definition of $C_k(\tt)$ we have that $\forall t\in]C_k(\tt),T_k(\tt)]$ $f(t)\ne \tfrac 13\gad$.
Then by the continuity of $f(t)$, $\forall t\in]C_k(\tt),T_k(\tt)]$ $f(t)< \tfrac 13\gad$.
Now it is enough to see that by item~\eqref{stcs2}, $[S_k(\tt),T_k(\tt)]\subset ]C_k(\tt),T_k(\tt)]$.

 \hfill$\triangle$

 Let
$$
B_k(\tt):=\sup\big\{\;t<C_k(\tt)\;\big|\; d(\psi_t(\theta),\Ga(\re)) \le \tfrac14\gad \;\big\}.
$$ 
\begin{Claim}\label{cBkCk}
$$
[B_k(\tt),C_k(\tt)]\subset [T_{k+1}(\tt),S_k(\tt)].
$$
\end{Claim}
\noindent{\it Proof:} \quad

Let $f,\,g$ be as in~\eqref{fgtv}.
By the definition of $S_k(\tt)$ we have that $S_k(\tt)\ge C_k(\tt)$.
By the definition of $B_k(\tt)$ and~\eqref{rho14ga1}, we have that
\begin{equation}\label{gbkck}
g|_{]B_k,C_k[}\ge  f|_{]B_k,C_k[} > \tfrac 14 \gad > C(B+1)\rho.
\end{equation}
By the definition of $S_k(\tt)$ we have that 
\begin{equation}\label{gcksk}
g|_{]C_k,S_k[}> C(B+1)\rho.
\end{equation}
By the definition of $C_k(\tt)$ and the continuity of $f(t)$
we have that 
\begin{equation}\label{gck}
g(C_k(\tt))\ge f(C_k(\tt))=\tfrac 13\gad > C(B+1)\rho.
\end{equation}
Joining~\eqref{gbkck}, \eqref{gcksk} and \eqref{gck} we get that
$$
g|_{]B_k,S_k[}>C(B+1)\rho.
$$
By the definition of $T_{k+1}(\tt)$ this implies that $T_{k+1}(\tt)\le B_k(\tt)$.

\hfill$\triangle$

If $t\in[B_k(\tt),C_k(\tt)]$, by the definition of $B_k(\tt)$ we have that
$$
d(\psi_t(\tt),\Ga)\ge \tfrac 14 \gad .
$$
Then by \eqref{Crho},
\begin{equation}\label{tbkck}
t\in[B_k(\tt),C_k(\tt)] \quad\then\quad d(x(t),\pi\Ga) > \tfrac 14 \ogd.
\end{equation}
By the definition of $T_{k+1}(\tt)$ we have that   
\begin{equation}\label{dvgcb}
\forall t\in]T_{k+1}(\tt),S_k(\tt)[\qquad \text{either}\qquad d\big(\psi_t(\tt),\Ga(v(t))\big)> C(B+1)\rho
\end{equation}
or $d(\psi_t(\tt),\Ga(\re))>\eta>C\rho$ 
(when $v(t)$ does not exist).
Here $\eta>C\rho$ follows from \eqref{rho14ga1}, \eqref{dgad}, \eqref{gadeta}.
The arguments in~\eqref{dyvrv}-\eqref{caseta3}
apply in the case~\eqref{dvgcb} to obtain
\begin{equation}\label{tk1sk}
t\in]T_{k+1}(\tt),S_k(\tt)[\quad\then\quad
d(\psi_t(\tt),\Ga)> C\rho.
\end{equation}
The continuity of $f$ and the definition of 
$B_k$ and $C_k$ give 
\begin{equation}\label{fbkck}
f(B_k) \le \tfrac 14\,\gad,\qquad f(C_k) =\tfrac 13\gad.
\end{equation}
From Claim~\ref{cBkCk}, \{\eqref{tbkck}, \eqref{defphi}\}, \eqref{cLp1},
 \{\eqref{gatima}, \eqref{fbkck}, \eqref{tima}\},
 \{\eqref{rho3}, \eqref{T>1}\} and 
 \{\eqref{tk1sk}, \eqref{phirho}\}, we have that
\begin{align}
\int_{T_{k+1}(\tt)}^{S_k(\tt)}\Big(\phi +c(L+\phi)\Big)
&\ge 
\int_{B_k(\tt)}^{C_k(\tt)}\Big( \tfrac 1{32}\, \e\, (\ogd)^{2} - \tfrac 1T \,\de^2\,\gad^2\Big)
+ \int_{[T_{k+1},S_k]\setminus[B_k,C_k]} \Big( \phi +c(L+\phi) \Big)
\notag\\
&\ge\big( \tfrac 1{32}\, \e\, (\ogd)^{2} - \tfrac 1T \,\de^2\,\gad^2\big)\, \a + 0.
\label{itk1sk}
\end{align}

Recall from \eqref{defLL} that
\begin{equation*}
\L:=L+\phi +c(L+\phi) -du,
\end{equation*}
where $u$ is from~\eqref{Ldu}.
Observe that the lagrangian flow for $\L$ is the same 
as the lagrangian flow $\psi_t$ for $L+\phi$.
Also $\cN(\L)=\cN(L+\phi)$ and $\cA(\L)=\cA(L+\phi)$.
Using \eqref{Ldu} and~\eqref{itk1sk},
\begin{align}
\int_{T_{k+1}(\tt)}^{S_k(\tt)} \L(\psi_t(\tt))\,dt
&= \int_{T_{k+1}(\tt)}^{S_k(\tt)}( L-du )
+\int_{T_{k+1}(\tt)}^{S_k(\tt)} \Big(\phi +c(L+\phi)\Big)
\notag
\\
&\ge 0+ \left(  \tfrac 1{32}\, \e\, (\ogd)^{2} -\tfrac 1T \,\de^2\,\gad^2\right) \a.
\label{farp}
\end{align}

\noindent
{\it Case 1:}
{\it Suppose that $T_k(\tt)-S_k(\tt)>T+2$.}

Let $m_k\in\na$ be such that
$$
S_k(\tt)+m_k T\le T_k(\tt)-1< S_k(\tt)+(m_k+1) T.
$$
Then $m_k\ge 1$.
Let $R_k(\tt):= S_k(\tt)+m_k T$.
Then $1\le T_k(\tt)-R_k(\tt) < T+1$.
By Claim~\ref{stcs}.\eqref{stcs6},
 $\Ga$ is $\tfrac\gad 3$-shadowed by $\psi_{[S_k,T_k]}(\tt)$. 
 Therefore by  inequality~\eqref{dextxy}  in Proposition~\ref{B71} there is $v\in\re$
  such that $ \forall t\in[S_k,T_k]$
 \begin{equation}\label{sktkd}
 d(\psi_t(\tt),\Ga(t+v))\le D\, \ee^{-\la\min\{t-S_k, T_k-t\}}
 [d(\psi_{S_k}(\tt),\Ga(S_k+v))
 +d(\psi_{T_k}(\tt),\Ga(T_k+v))].
 \end{equation}
 Also the choice of $v$ in Proposition~\ref{B71} is the same as in~\eqref{vttt}
 so that 
 \begin{equation}\label{tvvt}
 t+v = v(t) \qquad \forall t\in[S_k(\tt),T_k(\tt)].
 \end{equation}

\begin{figure}[h]
\resizebox*{8cm}{6cm}{\includegraphics{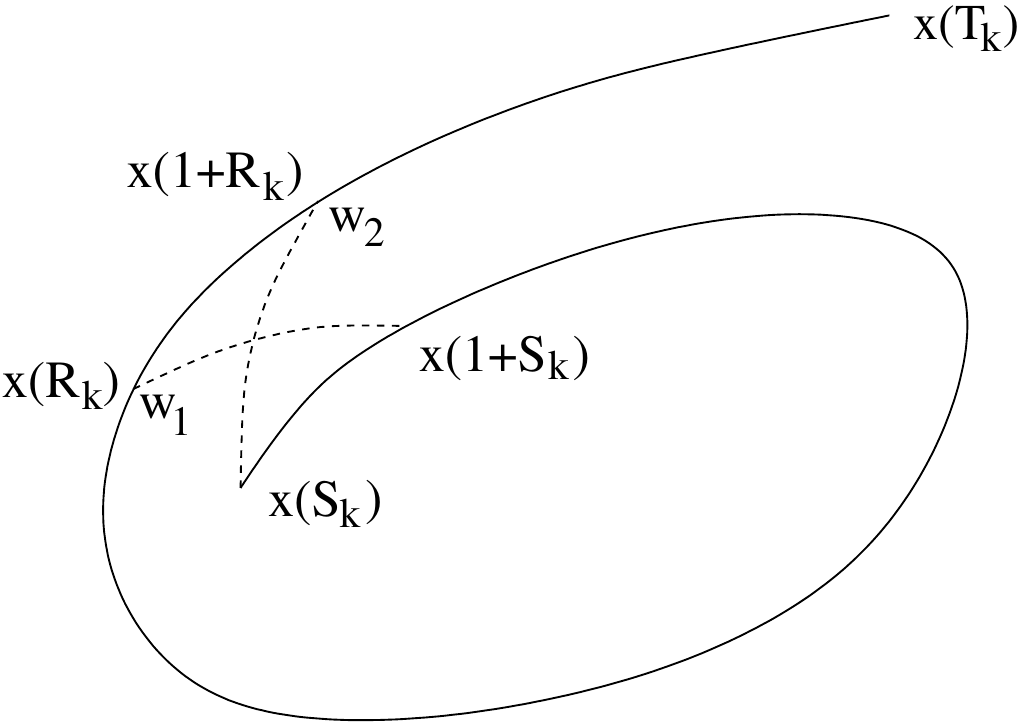}} 
\caption{The auxiliary segments $w_1$ an $w_2$.}
\label{near}
\end{figure}

By the definition of $S_k$ and $T_k$ in~\eqref{defSTCS} and the continuity of $g(t)$ on its domain we have that
\begin{equation}\label{sktkm}
g(S_k)\le C(B+1)\rho, \qquad
g(T_k)\le C(B+1)\rho.
\end{equation}

By~\eqref{sktkd}, \eqref{tvvt} and \eqref{sktkm} we have 
for $s\in [0,1]$ that
\begin{align*}
d\big(\psi_{s+R_k}(\tt) &,\Ga(v(s+R_k))\big)\le
\\
&\le 
 D \ee^{-\la\min\{s+R_k-S_k,T_k-s-R_k\}}
\big[d(\psi_{S_k}(\tt),\Ga(v(S_k))) + d(\psi_{T_k}(\tt),\Ga(v(T_k)))\big] 
\notag \\
&\le D \,\ee^0\;[g(S_k)+g(T_k)]
\le 2 DC (B+1) \rho.
\notag\\
d(\Ga(v(s+&S_k)),\psi_{s+S_k}(\tt)) \le 2 D C(B+1)\rho. 
\end{align*}
From~\eqref{tvvt} we have that 
$$
v(s+R_k)=s+R_k+v=s+S_k+v+ m_k T =v(s+S_k) + m_k T.
$$
So that $\Ga(v(s+R_k))=\Ga(v(s+S_k))$.
Adding the inequalities above we get
\begin{equation}
\forall s\in[0,1]\qquad d(\psi_{s+R_k}(\tt),\psi_{s+S_k}(\tt)) \le 4   DC(B+1)\rho.
\label{2cr}
\end{equation}

In local coordinates about $\pi(\Ga)$ define
\begin{alignat*}{6}
w_1(s+R_k) &= (1-&&s) \,&&x(s+R_k)  \;&+& &&s \, &&x(s+S_k), \qquad  s\in[0,1];
\\
w_2(s+S_k) &= &&s \, &&x(s+R_k) &+&\; (1-&&s)\, &&x(s+S_k),\qquad s\in [0,1].
\end{alignat*}
By Lemma~\ref{LKTr}\eqref{KTrb} and \eqref{2cr}
we have that
$$
A_{L+\phi}(x\vert_{[S_k,1+S_k]})+A_{L+\phi}(x|_{[R_k,1+R_k]})
\ge 
A_{L+\phi}(w_1)+A_{L+\phi}(w_2) -6 K D^2C^2 (B+1)^2\rho^2.
$$
Since the pairs of segments $\{\,x\vert_{[S_k,1+S_k]}, \, x|_{[R_k,1+R_k]}\,\}$
and $\{\, w_1,\,w_2\,\}$ have the same collections of endpoints
$$
\int_{S_k}^{1+S_k} du(\dx) + \int_{R_k}^{1+R_k} du(\dx)
= \oint_{w_1}du +\oint_{w_2}du.
$$
Therefore, since $c(L+\phi)$ is constant, 
\begin{equation}\label{crossingp}
A_\L(x\vert_{[S_k,1+S_k]})+A_\L(x|_{[R_k,1+R_k]})
\ge 
A_\L(w_1)+A_\L(w_2) -6 K  D^2 C^2 (B+1)^2 \rho^2.
\end{equation}

The integral of $d_xu$ on closed curves is zero. Therefore
\begin{equation}\label{cllo}
c(\L)=c(L+\phi+c(L+\phi))=0.
\end{equation}
Since $w_1*x|_{[1+S_k,R_k]}$ is a closed curve and $c(\L)=0$,  using \eqref{cloga},
\begin{equation}\label{closedp}
A_\L(w_1)+A_\L(x|_{[1+S_k,R_k]}) \ge 0.
\end{equation}
Using \eqref{Ldu} and \eqref{cLp1},
\begin{equation}\label{LL>}
\L = (L-du) + \phi + c(L+\phi)\ge 0+0 -\tfrac 1T \,\de^2(\gad)^2. 
\end{equation}
Since $T_k(\tt)-R_k(\tt)\le T+2$, 
using~\eqref{T>1}, on the curve $w_2*x|_{[1+R_k,T_k]}$
 we have that
\begin{equation}\label{restp}
A_\L(w_2)+A_\L(x|_{[1+R_k,T_k]}) \ge
 - \tfrac{T+2}T\, \de^2(\gad)^2 \ge - 3\, \de^2(\gad)^2.
\end{equation}
From \eqref{crossingp}, \eqref{closedp} and \eqref{restp} we get that
\begin{alignat*}{3}
A_\L(x|_{[S_k,T_k]})
&\ge A_\L(w_1) 
&&+A_\L(w_2)-6KD^2 C^2(B+1)^2\rho^2
\\
&  &&+A_\L(x|_{[1+S_k,R_k]})+A_\L(x|_{[1+R_k,T_k]})
\\
&\ge -6 KD&&^2C^2(B+1)^2\rho^2 - 3 \, \de^2(\gad)^2.
\end{alignat*}
{\it Case 2:  If $T_k-S_k\le T+2$,} from~\eqref{LL>} we also have
\begin{align*}
\hskip -2cm
A_\L(x|_{[S_k,T_k]})
&\ge - \tfrac{T+2}T \, \de^2(\gad)^2
\ge - 3\, \de^2(\gad)^2
\\
&\ge -6 KD^2C^2(B+1)^2\rho^2 - 3 \, \de^2(\gad)^2.
\end{align*}
Adding inequality \eqref{farp} and using \eqref{rho2} we obtain
a positive lower bound for the action independent of $k$:
$$
A_\L(x|_{[T_{k+1},T_k]}) 
\ge 
 \left(  \tfrac 1{32}\, \e\, (\ogd)^{2} -\de^2\,(\gad)^2\right) \a
  -12 K D^2C^2(B+1)^2\rho^2 - 3 \, \de^2(\gad)^2
    > 0.
$$

Since $x$ is  semi-static for $L+\phi$, and then also for $\L$, and by~\eqref{cllo}
$c(\L)=0$, the total action is finite:
 $$
 A_\L(x|_{]-\infty,0]})\le \max_{y,z\in M}\Phi^\L_{c(\L)}(y,z)<+\infty.
 $$ 
Therefore there must be at most finitely many $T_k$'s.

By item~\eqref{stcs3} in claim~\ref{stcs},  we have that $\a$-limit$(x,\dx)=\Ga$.
Since $\a$-limits of semi-static orbits are static (Ma\~n\'e \cite[theorem V.(c)]{Ma7}),
we obtain that $\Ga\subset\cA(L+\phi)$. This finishes the proof of proposition~\ref{Ppert}.

\hfill\qed

\appendix

\section{Shadowing}
\label{ashadowing}

  Let $\psi$ be the flow of a $C^1$ vector field on a compact manifold $M$.
  A compact $\psi$-invariant subset  $\La\subset M$
  is {\it  hyperbolic} for $\psi$ if the tangent bundle 
  restricted to $\La$ is decomposed as the Whitney sum
  $T_\La M = E^s \oplus E\oplus E^u$, where $E$ is the 1-dimensional 
  vector bundle tangent to the flow and there are constants $C,\la>0$ such that
  \begin{enumerate}[(a)]
  \item\label{hipa} $D\psi_t(E^s) = E^s$, $D\psi_t(E^u)=E^u$ for all $t\in\re$.
  \item\label{hipb} $| D\psi_t(v)| \le C\,\ee^{-\la t} |v|$ for all $v\in E^s$, $t\ge 0$.
  \item\label{hipc} $|D\psi_{-t}(u)|\le C\, \ee^{-\la t} |u|$ for all $u\in E^u$, $t\ge 0$.
  \end{enumerate} 
  It follows from the definition that the hyperbolic splittig 
  $E^s \oplus E\oplus E^u$ over $\La$ is continuous.

  From now on we shall assume that $\La$ does not contain fixed points
  for $\psi$.
  For $x\in \La$ define the following   stable and unstable sets:
  \begin{align}
  W^{ss}(x):&=\{\,y\in M\;|\; d(\psi_t(x),\psi_t(y))\to 0 \text{ as }t\to +\infty\,\},
  \notag\\
  W^{ss}_\e(x):&=\{\,y\in W^{ss}(x)\;|\;d(\psi_t(x),\psi_t(y))\le\e\;\;\forall t\ge 0\,\},
  \notag\\
  W^{uu}(x):&=\{\,y\in M\;|\; d(\psi_{-t}(x),\psi_{-t}(y))\to 0\text{ as } t\to +\infty\,\},
  \notag\\
  W^{uu}_\e(x):&=\{\,y\in W^{uu}(x)\;|\;d(\psi_{-t}(x),\psi_{-t}(y))\le \e\;\;\forall t\ge 0\,\},
  \label{wuue}\\
  \intertext\quad
  W^s_\e(x):&=\{\,y\in M\;|\;d(\psi_t(x),\psi_t(y))\le\e\;\;\forall t\ge 0\,\},
  \notag\\
   W^{u}_\e(x):&=\{\,y\in M\;|\;d(\psi_{-t}(x),\psi_{-t}(y))\le \e\;\;\forall t\ge 0\,\}.
   \notag
  \end{align}
  
  \medskip 
  
  Conditions~\{\eqref{hipa},\eqref{hipb},\eqref{hipc}\} are equivalent to \{\eqref{hipa},\eqref{hipd}\},
  where
  \begin{enumerate}[(a)]
  \addtocounter{enumi}{3}
  \item\label{hipd} There exists $T>0$ such that
   $\lV D\psi_T|_{E^s}\rV <\tfrac 12$ \quad and \quad
   $\lV D\psi_{-T}|_{E^u}\rV<\tfrac 12$.
  \end{enumerate}

  Let $\fX^k(M)$ be the Banach manifold of the $C^k$ vector fields on $M$, $k\ge 1$.
  Let $X=\partial_t\psi_t$ be the vector field of $\psi_t$. For $Y\in\fX^k(M)$ denote by
  $\psi^Y_t$ the flow of $Y$.

  \begin{Proposition}\label{unifhip}\quad

  There are open sets $X\in\cU\subset\fX^1(M)$ and $\La\subset U\subset M$ 
  such that for every $Y\in\cU$ the set $\La_Y:=\bigcap_{t\in\re}\psi^Y_t(\ov U)$
  is hyperbolic for the flow $\psi^Y_t$ of $Y$, with uniform constants $C$, $\la$, $T$
  on~\eqref{hipb}, \eqref{hipc} and \eqref{hipd}.
  \end{Proposition}
  
   Proposition~\ref{unifhip} can be proven by a characterization of hyperbolicity using
   cones (cf. Hasselblatt-Katok~\cite[Proposition 17.4.4]{HK}) and obtaining uniform 
   contraction (expansion) for a fixed iterate in $\La_Y$.  See
   Fisher-Hasselblatt \cite{FH} prop. 5.1.8 p. 256].

  \medskip

  \begin{mysec}{\bf Proposition {\cite[ 5.6, p.~63]{HPS}}, 
  {\cite[ 1.3]{Bowen6}}, {\cite[ 6.6.1]{FH}}.}
  \label{pHPS}\quad

  There are constants $C,\,\la>0$ such that, for small $\e$,
  \begin{enumerate}[(a)]
  \item $d\big(\psi_t(x),\psi_t(y)\big)\le C\, \ee^{-\la t}\, d(x,y)$
  when $x\in \La$, $y\in W_\e^{ss}(x)$, $t\ge 0$.
  \item $d\big(\psi_{-t}(x),\psi_{-t}(y)\big)\le C\, \ee^{-\la t}\, d(x,y)$
  when $x\in \La$, $y\in W_\e^{uu}(x)$, $t\ge 0$.

  \end{enumerate} 
  \end{mysec}
  
\medskip

  \begin{mysec}{\bf Canonical Coordinates \cite[3.1]{PS}, \cite[4.1]{HPS}, \cite[7.4]{Smale},
  \cite[1.4]{Bowen6}, \cite[1.2]{Bowen3}, \cite[6.2.2]{FH}:}
  \label{caco}
  \end{mysec}\vskip -5pt
  {\it There are $\a, \,\ga>0$ for which the following is true:
  If $x, y\in \La$ and $d(x,y)\le \a$ then there is a unique 
  $v=v(x,y)\in \re$ with $|v|\le \ga$ such that 
  \begin{equation}\label{ecaco}
  \langle x,y\rangle:=W_\ga^{ss}(\psi_v(x))\cap W^{uu}_\ga(y) \ne \emptyset.
  \end{equation}
  This set consists of a single point, which we denote 
  $\langle x,y\rangle\in M$. The maps $v$ and 
  $\langle\;,\;\rangle$ are continuous on the set
  $\{\,(x,y)\;|\; d(x,y)\le \a\,\}\subset \La\times\La$.
  }

     \begin{Lemma}\label{B4}

  There are $\eta_0>0$, $B>1$, and open sets $\La\subset U$, $X\in\cU\subset\fX^k(M)$  such that  
  \newline
  if $d(x,y)\le\eta_0$,  
  $Y\in\cU$, $x,\,y\in\La^Y_U$ and $\eta = B\, d(x,y)$  then
  \begin{gather}
  \langle x,y\rangle\in W^{ss}_\eta(\psi^Y_v(x))\cap W^{uu}_\eta(y)
  \qquad \text{with }\quad
  |v(x,y)| \le \eta \qquad 
  \label{vdxy}\\
  \text{and } \qquad d(x,\psi^Y_v(x))\le \eta.
  \label{Bdxy}
  \end{gather}
  \end{Lemma}
  
   \begin{proof}\quad

We have that $\langle x,x\rangle =x$ and $v(x,x)=0$.
By uniform continuity, given $\de>0$, for $d(x,y)$ small enough
\begin{equation}\label{dewx}
d(\langle x,y\rangle,x)\le \de,
\qquad
d(\langle x, y\rangle,y)\le \de,
\end{equation}
and $v=v(x,y)$ is so small that
\begin{equation}\label{dpsvd}
d(\psi_v(x),x)\le \de.
\end{equation}

 The continuity of the hyperbolic splitting implies that the angles
  $\measuredangle(E^s,E^u)$, $\measuredangle(Y,E^s)$ and 
  $\measuredangle(E^s\oplus\re Y,E^u)$
   are bounded away from zero, 
  uniformly on $\La^Y_V:=\bigcap_{t\in\re} \psi^Y_{-t}(\ov V)$, for some $V\supset U$
   and all $Y$ in an open set $\cU_0\subset\fX^1(M)$
  with $X\in\cU_0$.   
    There is $\be_1>0$ such that if $x,\,y\in\La^Y_U$ and $d(x,y)<\be_1$
  then 
  $$
  \langle x, y\rangle =W^s_\ga(x)\cap W^{uu}_\ga(y)\in V.
  $$

   The strong local invariant manifolds $W^{ss}_\ga$, $W^{uu}_\ga$ 
  are tangent to $E^s$, $E^u$ at $\La^Y_V$ and for a fixed $\ga$ as $C^1$ submanifolds
  they vary continuously on the base point $x\in M$ and on the vector field in the $C^1$
  topology (cf.  \cite[Thm. 4.3]{CroPo},\cite[Thm. 4.1]{HPS}). 
  There is a family of small cones $E^u_X(x)\subset C^u(x)\subset T_xM$, $E^s_X(x)\subset C^s(x)\subset T_xM$ 
  defined on a neighborhood $W$ of $\La$  invariant under $D\psi^Y_{-1}$ and $D\psi^Y_1$ respectively, for
  $Y$ in a $C^1$ neighborhood $\cW$ of $X$. The exponential of these cones contain $W^{uu}_{\ga}(x)$ and $W^{ss}_{\ga}(x)$ for 
  $x\in\La^Y_W$ and $Y\in\cW$. The angles between these cones are uniformly bounded 
  away from zero, so for example if $z^u\in W^{uu}(x)$, $z^s\in W^{ss}(x)$ and $d(z^u,x)$, $d(z^s,x)$ are small,
  then $d(z^u,x)+d(z^s,x)< A_0\,d(z^u, z^s)$ for some $A_0>0$.
  We can construct similiar cones separating $E^u$ from $E^s\oplus\re X$.
  
  Shrinking $U$ and $\cU$ if necessary there are $0<\be_2<\be_1$  and $A_1,\,A_2,\,A_3>0$
  such that if $Y\in\cU$, $x,\,y\in \La^Y_U$ and $d(x,y)<\be_2$,
  taking 
  $w:=\langle x,y\rangle\in W^s_\ga(x)\cap W^{uu}_\ga(y)$
  and $v$  such  that $w\in W^{ss}_\ga(\psi^Y_v(x))$,  i.e. 
$\psi^Y_v(x)\in\psi^Y_{[-1,1]}(x)\cap W^{ss}_\ga(w)$, then
  \begin{align}
d(x,w)+d(w,y)&\le A_1\, d(x,y),
 \label{a1xwy}
\\ 
d(x,\psi^Y_v(x))+d(\psi^Y_v(x),w)
&\le A_2\, d(x,w) \le A_2 A_1 \,d(x,y),
\label{a2xpw}
\\
|v|\le A_3\, d(x,\psi^Y_v(x))&\le A_3 A_2 A_1 \, d(x,y).
\notag
  \end{align}

  We can assume that $\cU_0$ and $U$ are so small that the constants
  $C$, $\la$, $\e$ in Proposition~\ref{pHPS} can be taken uniform for all 
  $Y\in\cU_0$ and in $\La^Y_U$.
    By Proposition~\ref{pHPS},
since $w:=\langle x,y\rangle \in W^{ss}_\ga(\psi_v(x))$, 
 we have that
\begin{align*}
\forall t\ge 0 \qquad
d\big(\psi^Y_t(\langle x,y\rangle),\psi^Y_{t}(\psi^Y_v(x))\big)
&\le C\,\ee^{-\la t}\,d(w,\psi_v^Y(x)) 
\\
&\le A_2 A_1 C\, \ee^{-\la t} \,d(x,y)
\qquad\text{ using~\eqref{a2xpw}.} 
\end{align*} 
Take $B_1:=(1+A_2)A_1 C$. Then if $d(x,y)<\be_2$ and 
$\eta= B_1\,d(x,y)$
 we obtain that
$\langle x,y\rangle\in W^{ss}_\eta(\psi^Y_v(x))$.

Since $w=\langle x,y\rangle\in W^{uu}_\ga(y)$ we have that 
\begin{align*}
\forall t\ge 0 \qquad
d(\psi^Y_{-t}(\langle x,y\rangle),\psi^Y_{-t}(y))
&\le C\, \ee^{-\la t} \,d(w,y) 
\\
&\le A_1C\, \ee^{-\la t} d(x,y)
\qquad\text{using }\eqref{a1xwy}.
\end{align*}
Thus if $\eta = B_1\, d(x,y)$ then $\langle x, y\rangle \in W^{uu}_\eta(y)$.

By~\eqref{dewx} and~\eqref{dpsvd}
there is $0<\eta_0<\be_2$ such that if $d(x,y)\le \eta_0$ then
$d(w,x)$, $d(w,y)$ and $d(\psi_v(x),x)$ are small enough
to satisfy the above inequalities.
Now let 
$$
B:=\max\{ 2,\, B_1,\, A_3 A_2 A_1,\, A_2 A_1\}.
$$

 \end{proof}

    \begin{Proposition}
    \label{B16}
  \quad
  
     There are open sets $X\in\cU\subset\fX^1(M)$
     and $\La\subset U\subset M$ and $\eta_0,\ga>0$, $B>1$ such that
     $$
     \forall \eta>0 \qquad \exists
      \be=\be(\eta)=\tfrac 1B\,\min\{\eta,\eta_0\} \qquad \forall Y\in\cU
     $$
   if $\psi_t=\psi^Y_t$  is the flow of $Y$, 
  $x, y \in \Om^Y_U:=\bigcap_{t\in\re}\psi_t(\ov U)$ and $s:\re\to\re$  continuous with $s(0)=0$ 
  satisfy
  \begin{equation}\label{csh}
  d(\psi_{t+s(t)}(y),\psi_t(x))\le\be\quad\text{ for }|t|\le L,
  \end{equation}
  then 
  \begin{equation}\label{setav}
  |s(t)|\le3\eta \quad \text{ for all }|t|\le L,
  \qquad
   |v(x,y)|\le\eta \quad \text{ and }
  \end{equation}
    \begin{align}
  \forall |s|\le L, \qquad
  &d(\psi_s(y),\psi_{s+v}(x))\le C\,\ee^{-\la(L-|s|)}\,
  \big[d(\psi_L(w),\psi_L(y))+d(\psi_{-L}(w),\psi_{-L+v}(x))\big],
  \notag\\
  &\text{where }\qquad w:=\langle x,y\rangle 
  =W^{ss}_\ga(\psi_v(x))\cap W^{uu}_\ga(y).
  \label{d<egdd1}
  \end{align}
  also
  \begin{equation}\label{dsh}
  \forall |s|\le L, \qquad
   d(\psi_s(y),\psi_s\psi_v(x))\le C\, \ga\,e^{-\la (L-|s|)}.
  \end{equation}
   In particular 
  $$
  d(y,\psi_v(x))\le C\, \ga\, e^{-\la L}.
  $$
  \end{Proposition}

 \begin{proof}
  
  Let $\ga=B \eta_0$ with $\{\eta_0, B\}$ 
  from~\ref{B4}.
  We may assume that $\eta$ is so small that
  \begin{gather}
  \eta < \tfrac{\ga}8,
  \label{etaga8}
  \\
  \sup\{\,d(\psi_u(x),x)\;:\; x\in M,\,|u|\le 4\eta\,\}\le \tfrac \ga 8.
  \label{u4eta}
  \end{gather}
  Let 
  \begin{equation}\label{beeta}
  \be=\be(\eta) := \tfrac 1B\, \min\{\eta,\eta_0\},
  \end{equation}
  where $B>1$ and $\eta_0$  are
  from lemma~\ref{B4}. Consider $x$, $y$ and $s(t)$ as
  in the hypothesis. Since $s(0)=0$ we have that  $d(x,y)\le \be$.
  Using lemma~\ref{B4} we can define
   \begin{equation}\label{wxyeta}
   w:=\langle x,y\rangle =W^{ss}_\eta(\psi_v(x))\cap W^{uu}_\eta(y)\ne\emptyset,
   \end{equation}
   we also have 
   \begin{equation}\label{vxyeta}
   |v|=|v(x,y)|\le \eta.
    \end{equation}
       Define the sets
   \begin{alignat*}{2}
   A&:=\{\,t\in[0,L]\;:\;|s(t)|\ge 3\eta\;&&\text{ or }
   \;d(\psi_t(y),\psi_t(w))\ge \tfrac 12{\ga}\,\},
   \\
   B&:=\{\,t\in[0,L]\;:\; |s(-t)|\ge 3\eta\;&&\text{ or }
    \;d(\psi_{-t+v}(x),\psi_{-t}(w))\ge \tfrac12\ga\,\}.
   \end{alignat*}
   
   Suppose that $A\ne \emptyset$. Let $t_1:=\inf A$. 
   Then 
   $d(\psi_t(y),\psi_t(w))\le \tfrac 12 \,\ga$, $\forall t\in[0,t_1]$.
   Since $w\in W^{uu}_\eta(y)$ and by~\eqref{etaga8}, $\eta<\tfrac 1{8}\ga$;
   from~\eqref{wuue} 
   we have that $d(\psi_t(y),\psi_t(w))\le \tfrac 18 \ga$, $\forall t\le 0$.
   Therefore
   \begin{equation}\label{wuut1}
   d(\psi_{t_1-r}(y),\psi_{t_1-r}(w))\le \tfrac 12 \,\ga, \qquad \forall r\ge 0.
   \end{equation}

   Since $s$ is continuous, $s(0)=0$ and $t_1\in\partial A$, we have that 
   $|s(t_1)|\le 3\eta$.
   Using~\eqref{u4eta} twice with $u=|s(t_1)|$, \eqref{wuut1} and the triangle inequality
   we obtain
   $$
   d(\psi_{t_1+s(t_1)-r}(y),\psi_{t_1+s(t_1)-r}(w))\le  \tfrac 34 \ga,
   \qquad \forall r\ge 0.
   $$
   Hence $\psi_{t_1+s(t_1)}(w)\in W^{uu}_\ga(\psi_{t_1+s(t_1)}(y))$. 
   From~\eqref{wxyeta},
   $w\in W^{ss}_\eta(\psi_v(x))$, and then
   \begin{equation}\label{uvga8}
   d(\psi_r(w),\psi_{r+v}(x))\le \eta<\tfrac \ga 8,
   \qquad \forall r\ge0.
   \end{equation}
   Since $|s(t_1)|\le 3\eta$, using~\eqref{u4eta} twice with $u=s(t_1)$, 
   and~\eqref{uvga8} with $r=t_1+p\ge 0$, and the triangle inequality,
   we get
   $$
   d(\psi_{t_1+s(t_1)+p}(w),\psi_{t_1+s(t_1)+v+p}(x))\
   \le\tfrac {3\ga}8, \qquad \forall p\ge 0.
   $$
   Hence $\psi_{t_1+s(t_1)}(w)\in W^{ss}_\ga(\psi_{s(t_1)+v}(\psi_{t_1}(x)))$.
   We have shown that
   \begin{equation}\label{pstst}
   \psi_{t_1+s(t_1)}(w)\in W^{ss}_\ga(\psi_{s(t_1)+v}(\psi_{t_1}(x)))
   \cap W^{uu}_\ga(\psi_{t_1+s(t_1)}(y)).
   \end{equation}
   Since $|s(t_1)+v|\le|s(t_1)|+|v|\le 4\eta<\ga$
   and by~\eqref{csh}, 
   \begin{equation}\label{lbe}
   d(\psi_{t_1+s(t_1)}(y),\psi_{t_1}(x))\le\be,
   \end{equation}
   equation~\eqref{pstst} implies that 
   \begin{gather*}
   v(\psi_{t_1}(x),\psi_{t_1+s(t_1)}(y))=s(t_1)+v(x,y),
   \\
   \psi_{t_1+s(t_1)}(w)=\langle \psi_{t_1}(x),\psi_{t_1+s(t_1)}(y)\rangle.
   \end{gather*}
   By Lemma~\ref{B4}, \eqref{lbe} and \eqref{beeta}, 
   \begin{gather}
   |s(t_1)+v|\le\eta \qquad \text{ and }
   \label{st1v}
   \\
   \psi_{t_1+s(t_1)}(w)\in W^{uu}_\eta(\psi_{t_1+s(t_1)}(y)),
   \;\text{in particular}
   \notag
   \\
   d(\psi_{t_1+s(t_1)}(w),\psi_{t_1+s(t_1)}(y))\le\eta.
   \label{dpstst}
   \end{gather}
   Since $|s(t_1)|\le 3\eta$, from~\eqref{u4eta},~\eqref{dpstst}
   and~\eqref{etaga8}, we get that
   $$
   d(\psi_{t_1}(w),\psi_{t_1}(y))\le \eta+2\left(\tfrac\ga 8\right)\le 
   \tfrac {3\ga}8.
   $$
   From~\eqref{st1v} and~\eqref{vxyeta} we have that
   $$
   |s(t_1)|\le|s(t_1)+v|+|v|\le 2\eta.
   $$
   These statements contradict $t_1\in A$.
   Hence $A=\emptyset$.
   
   Similarly one shows that $B=\emptyset$.
   Since $A=\emptyset$, inequality~\eqref{ywtga} holds for all $t\in[0,L]$.
   From~\eqref{wxyeta}, $w\in W^{uu}_\eta(y)$ and by~\eqref{etaga8}, 
   $\eta<\tfrac\ga 8$; thus inequality~\eqref{ywtga} also holds for $t\le 0$.
   \begin{equation}\label{ywtga}
   \forall t\le L\qquad d(\psi_t(y),\psi_t(w))< \tfrac 12{\ga}.
   \end{equation}
   Therefore
   \begin{equation}\label{psiLWuu}
    \psi_L(w)\in W^{uu}_{\frac 12\ga}(\psi_L(y)).
   \end{equation} 
   From Proposition~\ref{pHPS} we get
  \begin{equation*}
  \forall |s|\le L \qquad 
  d(\psi_s(w),\psi_s(y))\le C\,\ee^{-\la(L-|s|)}
  \,d(\psi_L(w),\psi_L(y)).
  \end{equation*}
  Similarly, $B=\emptyset$ imples that 
  \begin{equation}\label{psi-Lwss}
  \psi_{-L}(w)\in W^{ss}_{\frac 12\ga}(\psi_{-L+v}(x))
  \qquad \text{ and }
  \end{equation}
  $$
  \forall |s|\le L
  \qquad
  d(\psi_s(w),\psi_{s+v}(x))\le C\,\ee^{-\la(L-|s|)}\,
  d(\psi_{-L}(w),\psi_{-L+v}(x)).
  $$
   Adding these inequalities we obtain
  \begin{align}
  \forall |s|\le L \qquad
  &d(\psi_s(y),\psi_{s+v}(x))\le C\,\ee^{-\la(L-|s|)}\,
  \big[d(\psi_L(w),\psi_L(y))+d(\psi_{-L}(w),\psi_{-L+v}(x))\big],
  \notag\\
  &\text{where }\qquad w:=\langle x,y\rangle 
  =W^{ss}_\ga(\psi_v(x))\cap W^{uu}_\ga(y).
  \label{d<egdd}
  \end{align}
  This proves inequality~\eqref{d<egdd1}.
  
    From~\eqref{vxyeta}, $|v(x,y)|\le\eta$.
  The fact $A\cup B=\emptyset$ also gives $|s(t)|\le 3\eta$
  for $t\in[-L,L]$. This proves~\eqref{setav}.
  From~\eqref{psiLWuu}, \eqref{psi-Lwss} and~\eqref{d<egdd}
  we get inequality~\eqref{dsh}.
  
  \end{proof}

   \begin{Proposition}\label{B5}\quad
   
   Let $\ga$, $\eta_0$  and $\be=\be(\eta)$ be from Proposition~\ref{B16}.
   Given $\eta<\min\{\eta_0, \tfrac 12\ga\}$
   \begin{enumerate}[(a)]
   \item\label{B5a}
   If $x,\,y\in\La$ and $s:[0,+\infty[\to\re$ continuous
   with $s(0)=0$
   satisfy 
   $$
   d(\psi_{t+s(t)}(y),\psi_t(x))\le \be \qquad \forall t\ge0,
   $$
   then $|s(t)|\le 3\eta$ for all $t\ge 0$ and there is $|v(x,y)|\le\eta$
   such that $y\in W^{ss}_\ga(\psi_v(x))$.
   
   \item\label{B5b}
   Similarly, if $x,\,y\in\La$, $s:]-\!\infty,0]\to\re$ is continuous with $s(0)=0$
   and
   $$
   d(\psi_{t+s(t)}(y),\psi_t(x))\le\be \qquad \forall t\le 0,
   $$
    then $|s(t)|\le 3 \eta$ for all $t\le 0$ and there is $|v(x,y)|\le\eta$
   such that $y\in W^{uu}_\ga(\psi_v(x))$.
   \end{enumerate}
   \end{Proposition}

  \begin{proof}\quad
  
  We only prove item~\eqref{B5a}.
  The same proof as in Proposition~\ref{B16} shows that 
  taking
  $$
  w:=\langle x,y\rangle=W^{ss}_\eta(\psi_v(x))\cap W^{uu}_\eta(y)\ne \emptyset,
  $$
  we have that
  $|v|=|v(x,y)|\le\eta$ and 
  $$
   \emptyset=A:=\{\,t\in[0,+\infty[\;:\;|s(t)|\ge 3\eta\;\text{ or }
   \;d(\psi_t(y),\psi_t(w))\ge \tfrac 12 {\ga}\,\}.
  $$
  Therefore $|s(t)|\le 3\eta$ for all $t\ge 0$ and 
  $w\in W^{ss}_{\frac 12 \ga }(y)\cap W^{ss}_\eta(\psi_v(x))$.
  Since $\tfrac 12 \ga+\eta<\ga$ we get that
  $y\in W^{ss}_\ga(\psi_v(x))$.

  \end{proof}

  \medskip
  
   \begin{Proposition}\label{B71}\quad

         There are $D>0$, $\be_0>0$ and  open sets $X\in\cU\subset\fX^1(M)$,
     $\La\subset U\subset M$,  such that
     $$
     \forall \be\in]0,\be_0] \qquad \forall Y\in\cU,
     $$
   if $Y\in\cU$, $\psi_t=\psi^Y_t$  is the flow of $Y$, 
  $x,\, y \in \La^Y_U:=\bigcap_{t\in\re}\psi_t(\ov U)$ 
  and $s:\re\to\re$  continuous with $s(0)=0$ satisfy
  \begin{equation}\label{csh2}
  d(\psi_{t+s(t)}(y),\psi_t(x))\le\be\quad\text{ for }|t|\le L,
  \end{equation}
  then $|s(t)|\le D\be$ for all $|t|\le L$ and there is $|v|=|v(x,y)|\le D\be$ 
  such that
     \begin{align*}
  \forall |s|\le L, \qquad
  d(\psi_s(y),\psi_{s+v}(x))\le D\,\be\,\ee^{-\la(L-|s|)}.
  \end{align*}
  
  Moreover for all $|s|\le L$,
   \begin{equation}\label{dextxy}
  d(\psi_s(y),\psi_{s+v}(x))\le
   D\,\ee^{-\la(L-|s|)}\,
  \big[ d(\psi_L(y),\psi_{L+v}(x)) + d(\psi_{-L}(y),\psi_{-L+v}(x)) \big],
  \end{equation}
  and $v$ is determined by
  $$
  \langle x,y\rangle =W^{ss}_\ga(\psi_v(x))\cap W^{uu}_\ga(y)\ne \emptyset.
  $$

  \end{Proposition}

    \begin{proof}\quad
  
  Let $C$, $\cU$, $U$ $\eta_0>0$ and $B$ be from Proposition~\ref{B16}.
  The continuity of the hyperbolic splitting implies that the angle
  $\measuredangle(E^s,E^u)$ is bounded away from zero.
  As in the argument after \eqref{dpsvd}, there are invariant families of cones separating
  $E^s$ from $E^u$ whose image under the exponential map contain the local invariant 
  manifolds $W^{ss}_\ga$, $W^{uu}_\ga$.
  And hence as in~\eqref{a1xwy}
  there are $A,\,\be_1>0$ such that if $x,\, y\in \La^Y_U$,
  $d(x,y)<\be_1$ and 
  $$
  w=\langle x,y\rangle =W^{ss}_\ga(\psi_v(x))\cap W^{uu}_\ga(y),
  $$
  then 
  \begin{equation}\label{wpsiv}
  d(w,\psi_v(x))+d(w,y) \le A\, d(\psi_v(x),y).
  \end{equation}
  Suppose that $0<\be<\min\{\tfrac 1B\eta_0,\,\be_1\}$
  and 
  $x$, $y$, $s(t)$, $\psi^Y_t$, $L$ satisfy~\eqref{csh2}.
  Apply Proposition~\ref{B16} with $\eta:= B\be$.
  
  Then $|s(L)|\le 3\eta$, 
  and 
  \begin{align*}
  d(\psi_L(y),\psi_L(x))
  &\le d(\psi_{L+s(L)}(y),\psi_L(x))+|s(L)| \cdot \Vert Y\Vert_{\sup}
  \\
  &\le \be + 3\eta \lV Y\rV_{\sup} < \a,
  \end{align*}
  if $\be$ is small enough.
  So that $\langle \psi_L(x),\psi_L(y)\rangle$ is well defined.
  Similarly $|s(-L)|\le 3\eta$ and
  $d(\psi_{-L}(y),\psi_{-L}(x))<\a$.
  Since the time $t$ map $\psi_t$ preserves the family of
  strong invariant manifolds,
  in equation~\eqref{d<egdd1} we have that 
  \begin{align*}
  \psi_L(w) &=\langle \psi_{L}(x),\psi_L(y)\rangle =
  W^{ss}_\ga(\psi_{L+v}(x))\cap W^{uu}_\ga(\psi_L(y)),
  \\
    \psi_{-L}(w) &=\langle \psi_{-L}(x),\psi_{-L}(y)\rangle =
  W^{ss}_\ga(\psi_{-L+v}(x))\cap W^{uu}_\ga(\psi_{-L}(y)).
  \end{align*}
  Therefore, using~\eqref{wpsiv},
  \begin{align}
  d(\psi_L(w),\psi_L(y))+d&(\psi_{-L}(w),\psi_{-L+v}(x))
  \notag\\
  &\le  A \big[ d(\psi_{L+v}(x),\psi_L(y)) + d(\psi_{-L+v}(x),\psi_{-L}(y)) \big],
   \label{wxyd}
   \\
  d(\psi_{L+v}(x),\psi_{L}(y)) &\le
  d(\psi_{L+v}(x),\psi_{L}(x))+d(\psi_{L}(x),\psi_{L+s(L)}(y))
  +d(\psi_{L+s(L)}(y),\psi_{L}(y))
  \notag \\
  &\le |v| \lV Y\rV_{\sup}+\be+ |s(L)|\,\lV Y\rV_{\sup}
  \notag\\
  &\le B_1 \be,
  \notag
  \end{align}
  for some $B_1=B_1(\cU)>0$, because by Proposition~\ref{B16},
  $|v|\le \eta$, $|s(t)|\le 3\eta$ and $\eta = B \be$, so that
  $$
  |v| \le B\be,\qquad |s(t)|\le 3 B \be.
  $$
  A similar estimate holds for $d(\psi_{-L+v}(x),\psi_{-L}(y))$
  and hence from~\eqref{wxyd},
  $$
  d(\psi_L(w),\psi_L(y))+d(\psi_{-L}(w),\psi_{-L+v}(x))
  \le  2AB_1 \,\be.
  $$
  Replacing this  in~\eqref{d<egdd1} we have that
  $$
  \forall |s|\le L,\qquad
  d(\psi_s(y),\psi_{s+v}(x))\le D_1\,\be \,\ee^{-\la(L-|s|)},
  $$
  where $D_1=2 A B_1 C$.
  
  By~\eqref{wxyd} and~\eqref{d<egdd1} we also have that 
  $$
  d(\psi_s(y),\psi_{s+v}(x))\le
  AC\,\ee^{-\la(L-|s|)}\,
  \big[ d(\psi_L(y),\psi_{L+v}(x)) + d(\psi_{-L}(y),\psi_{-L+v}(x)) \big].
  $$
  Now take $D:=\max\{ D_1,\,B,\,3B,\,AC\,\}$.
  
  \end{proof}

 \begin{Definition}\label{dfe}\quad
  
  We say that $\psi|_\La$ is {\it flow expansive} if for every 
  $\eta>0$ there is $\ov\a=\ov\a(\eta)>0$ such that 
  if $x\in\La$, $y\in M$ and  there is $s:\re\to\re$ continuous
  with $s(0)=0$ and $d(\psi_{s(t)}(y),\psi_t(x))\le\ov \a$ for all $t\in\re$,
  then
  $y=\psi_v(x)$ for some 
  $|v|\le \eta$.
  \end{Definition}

  \begin{Remark}\label{rue}\quad
  
  Observe that Proposition~\ref{B16} implies 
  uniform expansivity in a neighbourhood of $(X,\La)$, namely
  there are neighbourhoods $X\in\cU\subset\fX^1(M)$ and $\La\subset U\subset M$
  such that for every $\eta>0$ there is $\a=\a(\eta,\cU,U)>0$ such that if
  $x\in\La^Y_U:=\cap_{t\in\re}\psi^Y_t(\ov U)$, $y\in M$, $s:(\re,0)\to (\re,0)$ 
  continuous and $\forall t\in \re$, $d(\psi^Y_{s(t)}\big(y),\psi^Y_t(x)\big)<\a$;
  then $y=\psi^Y_v(x)$ for some $|v|<\eta$.
  See also Fisher-Hasselblatt \cite[cor. 5.3.5]{FH}.
  
  This also implies uniform h-expansivity of their time-one maps as in Definition~\ref{duhe}.
  \end{Remark}

  \begin{Definition}\label{hexpan}\quad
    
  Let $f:X\to X$ be a homeomorphism. For $\e>0$ and $x\in X$ define
  $$
  \Ga_\e(x,f):=\{\,y\in X\;|\;\forall n\in \Z\quad  d(f^n(y),f^n(x))\le \e \,\}.
  $$
  We say that $f$ is {\it entropy expansive} or {\it h-expansive} if there
  is $\e>0$ such that
  $$
  \forall x\in X\qquad h_{\text{top}}(\Ga_\e(x,f),f)=0.
  $$
  Such an $\e$ is called an h-expansive constant for $f$.
    \end{Definition}

 \begin{Definition}\label{duhe}\quad
   
   Let $\cU$ be a topological subspace of $C^0(X,X)\supset\cU$ and 
   $Y\subseteq X$ compact.
   We say that $\cU$ is {\it uniformly h-expansive} on $Y$
   if there is $\e>0$ such that 
   $$
   \forall f\in\cU\quad \forall y\in Y
   \qquad
   h_{\text{top}}(\Ga_\e(y,f),f)=0.
   $$
   In our applications $\cU$ will be a $C^1$ neighbourhood of a diffeomorphism
   endowed with the $C^0$ topology. An h-expansive homeomorphism corresponds
   to $\cU=\{ f\}$.
   \end{Definition}

  \begin{Definition}\label{B8}\quad
  
  Let $L>0$, we say that $(T,\Ga)$ is an $L$-specification if
  \begin{enumerate}[(a)]
  \item $\Ga=\{x_i\}_{i\in\Z}\subset \La$.
  \item $T=\{t_i\}_{i\in\Z}\subset\re$\quad and\quad  $t_{i+1}-t_i\ge L$\; $\forall i\in\Z$.
  \end{enumerate}
  We say that the specification $(T,\Ga)$ is $\de$-possible if
  $$
  \forall i\in\Z\qquad d(\psi_{t_i}(x_i),\psi_{t_i}(x_{i-1}))\le \de.
  $$ 
  \end{Definition}

\begin{Theorem}\label{SHL}\quad

    Given $\ell>0$ there are $\de_0=\de_0(\ell)>0$ and $E=E(\ell)>0$ 
    such that if $0<\de<\de_0$ and
  $(T,\Ga)=(\{t_i\},\{x_i\})_{i\in\Z}$ is a $\de$-possible $\ell$-specification on $\La$ then
  there exist $y\in M$ and $\si:\re\to\re$ continuous, piecewise linear,  strictly increasing with 
  $\si(t_0)=t_0$ and $|\si(t)-t| < E\,\de$  such that
  \begin{equation}\label{shlip}
  \forall i\in \Z \quad
   \forall t\in]t_i,t_{i+1}[
   \qquad
  d\big(\psi_{\si(t)}(y),\psi_t(x_i)\big)< E\,\de.
  \end{equation}
  Moreover, if the specification is periodic then $y$ is a periodic point for $\psi$.
\end{Theorem}

 Theorem~\ref{SHL} does not
need that the hyperbolic set $\La$ is locally maximal, but if not the point $y$ is not in $\La$.
In Fisher-Hasselbaltt~\cite{FH}  the shadowing theorem~\ref{SHL} is proved without
a local maximality assumption and with the Lipschitz estimate \eqref{shlip} and $\si(t)$ a
homeomorphism such that $\si(t)-t$  has Lipschitz constant $E\de$.
But then proposition~\ref{B71} above proves the bound $|\si(t)-t|<E\de$ and moreover,
that $\si(t)-t$ can be taken constant on each interval $]t_i,t_{i+1}[$.

Theorem~\ref{SHL} is proved in Bowen~\cite{Bowen6} (2.2) p.~6 with $\si(t)-t$  constant on each 
$]t_i,t_{i+1}[$ and without the estimate $E\de$.
  A proof of  theorem~\ref{SHL} for flows without the local maximality hypothesis
  and with the explicit estimate  $E\de$ appears in Palmer \cite{Palmer}
  theorem~9.3, p. 188. In \cite{Palmer}, \cite{Palmer2009} the theorem 
  requires an upper bound on the lengths of the intervals in $T$. 
  This is because there the theorem is proven also for perturbations 
  of the flow. Indeed by Proposition~\ref{B71} longer intervals in 
  $T$  {\sl improve} the estimate.

\nocite{Be3}


\def\cprime{$'$} \def\cprime{$'$} \def\cprime{$'$} \def\cprime{$'$}
\providecommand{\bysame}{\leavevmode\hbox to3em{\hrulefill}\thinspace}
\providecommand{\MR}{\relax\ifhmode\unskip\space\fi MR }
\providecommand{\MRhref}[2]{%
  \href{http://www.ams.org/mathscinet-getitem?mr=#1}{#2}
}
\providecommand{\href}[2]{#2}

\end{document}